\documentclass[11pt]{article}

\usepackage{amssymb}
\usepackage{latexsym}
\usepackage{amsmath}
\usepackage{amscd}
\usepackage{amsbsy}
\usepackage[pdftex]{graphicx}
\usepackage{graphicx}
\usepackage{dsfont}
\usepackage[square,sort,comma,numbers]{natbib} 

\newtheorem{theorem}{Theorem}

\newtheorem{proposition}{Proposition}
\newtheorem{remark}{Remark}

\newenvironment{proof}[1][Proof]{\textbf{#1.} }{\ \rule{0.5em}{0.5em}}

\newcommand{\Var}{\text{Var}}

\newcommand{\Card}{\text{Card}}

\def\ds{\displaystyle}

\begin{document}

\title{Adaptive test for large  covariance matrices with missing observations}

\author{Cristina Butucea$^{1,2}$, Rania Zgheib$^{1,2}$ \\
$^1$ Universit\'e Paris-Est Marne-la-Vall\'ee, \\
LAMA(UMR 8050), UPEM, UPEC, CNRS,
F-77454, Marne-la-Vall\'ee, France\\
$^2$ ENSAE-CREST-GENES \\
3, ave. P. Larousse
92245 MALAKOFF Cedex, FRANCE
} 

\date{}

\maketitle

\begin{abstract}
We observe $n$ independent $p-$dimensional Gaussian vectors with missing coordinates, that is
each value (which is assumed standardized) is observed with probability $a>0$. We investigate the problem of minimax nonparametric testing that the high-dimensional covariance matrix $\Sigma$ of the underlying Gaussian distribution is the identity matrix, using these partially observed vectors. Here, $n$ and $p$ tend to infinity and $a>0$ tends to 0, asymptotically.

We assume that $\Sigma$ belongs to a Sobolev-type ellipsoid with parameter $\alpha >0$. When $\alpha$ is known, we give asymptotically minimax consistent test procedure and find the minimax separation rates $\tilde \varphi_{n,p}=  (a^2n \sqrt{p})^{- \frac{2 \alpha}{4 \alpha +1}}$, under some additional constraints on $n,\, p$ and $a$. We show that, in the particular case of Toeplitz covariance matrices,the minimax separation rates are faster, $\tilde \phi_{n,p}= (a^2n p)^{- \frac{2 \alpha}{4 \alpha +1}}$. We note how the "missingness" parameter $a$ deteriorates the rates
with respect to the case of fully observed vectors ($a=1$).

We also propose adaptive test procedures, that is free of the parameter $\alpha$ in some interval, and show that the loss of rate is $(\ln \ln (a^2 n\sqrt{p}))^{\alpha/(4 \alpha +1)}$ and $(\ln \ln (a^2 n p))^{\alpha/(4 \alpha +1)}$ for Toeplitz covariance matrices, respectively.
\end{abstract}


\\

\section{Introduction}
Recently,  problems related to high-dimensional data became more popular. In particular, in many areas such as genetics, meteorology and others, the generated  data sets are high-dimensional and incomplete, in the sense that they contain missing values. In this paper we investigate the problem of testing large covariance matrices from a sample of partially observed vectors.

Let $X_1,\dots,X_n$, be $n$ independent and identically distributed $p$-vectors following a  multivariate normal distribution $\mathcal{N}_p(0,\Sigma)$, where $\Sigma=[\sigma_{ij} ]_{ 1\leq i,j \leq p}$ is the normalized covariance matrix, with $\sigma_{ii} = 1$, for all $i=1$ to $p$.
Let us denote $X_k = (X_{k,1},\dots, X_{k,p})^\top$ for all $k=1,\dots,n$. Let $ \{ \varepsilon_{k,j} \}_{ 1 \leq k \leq n , 1 \leq j \leq p }$ be a sequence of i.i.d. Bernoulli random variables with parameter $a \in (0,1)$, $\mathcal{B}(a)$, and independent from $ X_1, \dots, X_n$. We observe $n$ i.i.d. random vectors $Y_1, \dots, Y_n$ such that 
$$
Y_k = ( \varepsilon_{k,1} \cdot X_{k,1}, \dots ,  \varepsilon_{k,p} \cdot X_{k,p} )^\top  \quad \text{ for all }   k=1,\dots,n.
$$
Each component of the vector $Y_k$ is observed with probability equal to $a$ and this is the context of missing observations.
We denote by $P_{a , \Sigma}$ the  probability distribution of the random vector $Y_k$ when $X_k \sim \mathcal{N}_p(0,\Sigma)$ and $\varepsilon_{k,j} \sim \mathcal{B}(a)$. We also denote by $\mathbb{E}_{a,\Sigma}$ and $\Var_{a,\Sigma}$  the expected value and the variance  associated to $P_{a , \Sigma}$.
Given the partially observed vectors $Y_1, \dots, Y_n$,  we want to test the null hypothesis 
\begin{equation}
\label{H0}
H_0 : \Sigma =I
\end{equation}
against a composite alternative hypothesis 
\begin{equation}
\label{H1}
H_1 : \Sigma \in Q(\mathcal{G}(\alpha), \varphi)
\end{equation}
where $\alpha >0$ and $\mathcal{G}(\alpha) $ is either 
\[
\mathcal{F}(\alpha) = \{\Sigma > 0,~ \mbox{symmetric} ;  \ds\frac{1}{p}\ds\sum_{1 \leq i<j \leq p} \sigma_{ij}^2 |i-j|^{2 \alpha}
     \leq 1 \text{ for all $p \geq 1$ and }   \sigma_{ii} = 1 ~\text{for all}~ 1 \leq i \leq p \}
\]
in the general case or 
\[
\mathcal{T}(\alpha) = \{ \Sigma >0,~ \mbox{symmetric}  , \Sigma \text{ is Toeplitz } ; \ds\sum_{j \geq 1} \sigma_j^2 j^{2\alpha} \leq 1 \text{  and } \sigma_0 = 1 \} 
\]
for the case of Toeplitz matrices. Thus, we define the following $\ell_2$ ellipsoids with $\ell_2$ balls removed:
\begin{equation} 
\label{classe generale}
Q(\mathcal{F}(\alpha), \varphi) = \Big\{ \Sigma \in \mathcal{F}(\alpha) \quad \text{ such that } \ds\frac 1{p} \sum_{ 1 \leq i < j \leq p} \sigma_{ij}^2  \geq \varphi^2\Big\}
\end{equation}
and 
\begin{equation}
\label{classe toeplitz}
Q(\mathcal{T}(\alpha), \varphi) = \Big\{ \Sigma \in \mathcal{T}(\alpha) \quad \text{ such that }  \sum_{  j \geq 1} \sigma_{j}^2  \geq \varphi^2 \Big\}
\end{equation}
Typically, the test procedures depend on the parameter $\alpha$ and  it is therefore useful to construct a test procedure that is adaptive to  $\alpha $  in some interval. Here we propose minimax and adaptive procedures for testing in the context of missing observations. 

The problem of estimating a covariance matrix of partially observed vectors was investigated several times in the literature. The simplest method to deal with missing data is to ignore the missing values and restrict the study to a subset of fully observed variables. This  method is not always reliable mainly when the number of missing values is relatively high. Hence, in order to treat this problem, methods based 
 on filling in the missing values were developed, in particular the Expectation-Maximization(EM) algorithm see \cite{Schneider01}. Recently, \citep{Lounici14} proposed an estimating procedure that does not need imputation of the missing values. Instead, the setup with missing values is treated as an inverse problem. We will also follow this approach for the test problem.

The problem of testing large covariance matrices  was considered only in the case of complete data. Out of the large amount of results in the literature on this latter  problem, we mention only the most related papers where procedures to test the null hypothesis $H_0 $  in \eqref{H0} are derived. We refer to  \citep{Bai2009} , \cite{JiangJiangYang12} and \cite{WangCaoMiao13}, where test procedures based  on the likelihood ratio  are proposed, and to \cite{LedoitWolf02}, \cite{Srivastava2005}, \cite{ChenZZ10} and \cite{CaiMa13},  where test statistics   based on  the quadratic loss function $tr(\Sigma-I)^2$ are used. Note that in \cite{ButuceaZgheib2014A} and \cite{ButuceaZgheib15}  asymptotically  consistent  test procedures where given in order to test \eqref{H0} against \eqref{H1}, when the covariance matrices belongs to \eqref{classe generale} and to \eqref{classe toeplitz}, respectively. They describe the minimax and sharp minimax separation rates. Here,  we give the minimax separation rates when assuming that we have partially observed vectors. We describe how the "missingness" parameter $a$ deteriorates the minimax rates in this context. Moreover we develop  consistent test procedures free of the class parameter $\alpha$, via an aggregation procedure of tests.

Missing observations appeared  recently in random matrix theory, see \cite{JurczakRohde15}. They show that the sequence of the spectral measures of  sample  covariance matrices with missing observations converge weakly to a sequence of  non random measures. Also they studied the limits of the extremes eigenvalues in the same context.

In this paper, we  describe the  minimax separation rate for testing $H_0$ given in \eqref{H0} against the composite alternative  $H_1$ in \eqref{H1}, when the data contains missing values. 
For a test procedure $\Delta$ we define 
the type I error probability by $\eta(\Delta)= P_I( \Delta =1)$, 
the maximal type~II error probability by $ \beta( \Delta, Q(\mathcal{G}(\alpha), \varphi))= \sup\limits_{\Sigma \in Q(\mathcal{G}(\alpha), \varphi)} P_\Sigma(\Delta =0) $ and the total error probability by
\[
\gamma(\Delta, Q(\mathcal{G}(\alpha), \varphi)) = \eta(\Delta) + \beta( \Delta, Q(\mathcal{G}(\alpha), \varphi)).
\]
Moreover, we define the minimax total error probability over the class $Q(\mathcal{G}(\alpha), \varphi)$ by
\[
\gamma :=  \inf\limits_\Delta \gamma(\Delta, Q(\mathcal{G}(\alpha), \varphi))
\]
where the infimum is taken over all  possible test procedures.
We define the minimax separation rate $\widetilde\varphi_\alpha$. On the one hand, we construct  a test procedure $\Lambda$ and derive  the conditions  on  $\varphi$ for which   $\gamma(\Lambda, Q(\alpha, \varphi)) \to 0 $. The test $\Lambda$ will be called asymptotically minimax consistent. On the other hand we give the conditions on $\varphi$  for which  $\gamma \to 1$. The previous conditions together allow us to determinate the minimax separations rate $\widetilde\varphi_\alpha$, such that there exists the test $\Lambda$ with
\[
\gamma(\Lambda, Q(\mathcal{G}(\alpha), \varphi)) \to 0 \quad \text{ if } \ds\frac{\varphi}{\widetilde\varphi_\alpha} \to + \infty,
\]
and
\[
\gamma =  \inf\limits_\Delta \gamma(\Delta, Q(\mathcal{G}(\alpha), \varphi)) \to 1 \quad \text{ if } \ds\frac{\varphi}{\widetilde\varphi_\alpha} \to 0.
\]
In other words, when $\varphi >> \widetilde\varphi_\alpha$ there exists an asymptotically minimax consistent test procedure and when $\varphi << \widetilde\varphi_\alpha$, there is no asymptotically consistent test procedure which can  distinguish between the null and the alternative hypothesis.

We also consider the problem of  adaptation with respect to the parameter $\alpha$. To treat this problem we first  assume that $\alpha \in A$, for $A$ an interval, and define  a larger class of matrices under the alternative than \eqref{H1}. The testing problem we are interested in now, is to test $H_0$ in \eqref{H0} against
\[
H_1 : \Sigma \in \underset{\alpha \in A}{ \cup} Q( \mathcal{F}(\alpha), \mathcal{C} \psi_\alpha) \, ,
\]
where $\psi_\alpha = \rho_{n,p}/ \widetilde\varphi_\alpha$, and $ \widetilde\varphi_\alpha $ is the minimax separation rate of testing $H_0$ given in \eqref{H0} against $H_1$ in \eqref{H1} for a known $\alpha$.   
Our aim is to  construct  a test procedure $\Delta_{ad}$ and to find the loss $\rho_{n,p}$  such that for a large enough constant $\mathcal{C}>0$:
\[
\gamma(\Delta_{ad},  \underset{\alpha \in A}{ \cup} Q( \mathcal{F}(\alpha), \mathcal{C} \psi_\alpha)) \to 0 .
\]
In this case we say that $\Delta_{ad}$ is an  asymptotically adaptive consistent test procedure.

The paper is structured as follows: in section~\ref{sec:general} we solve the case of general covariance matrices in $\mathcal{F}(\alpha)$ and in section~\ref{sec: Toeplitz matrices} the particular case of Toeplitz covariance matrices in $\mathcal{T}(\alpha)$. In section \ref{sec : separation rate general case}, we study the test problem with alternative hypothesis $Q(\mathcal{F}(\alpha), \varphi)$. We construct an asymptotically minimax consistent test procedure based on the data with missing observations and show that the minimax separation rate is
$$
\tilde \varphi_\alpha(\mathcal{F}) = (a^2n \sqrt{p})^{- \frac{2 \alpha}{4 \alpha +1}}.
$$   
In section~\ref{sec: adaptivity general case}, we propose a test procedure adaptive to the unknown parameter $\alpha$.
In section~\ref{sec: Toeplitz matrices}, we study the problem  with alternative hypothesis $Q(\mathcal{T}(\alpha), \varphi)$ and derive analogous results. The minimax separation rate is
$$
\tilde \phi_\alpha(\mathcal{T}) = (a^2 n p)^{- \frac{2 \alpha}{4 \alpha +1}}.
$$
We can view the vectors $X_k$ in this case as a sample of size $p$ from a stationary Gaussian process. However, due to the missing data, this is not true anymore for vectors $Y_k$. Minimax and adaptive rates of testing are faster by a factor $\sqrt{p}$
over classes $\mathcal{T}(\alpha)$ than over the classes $\mathcal{F}(\alpha)$. Note that the adaptive procedure attains the rates $(\sqrt{\ln \ln (a^2 n\sqrt{p})}/(a^2 n\sqrt{p}))^{2\alpha/(4 \alpha +1)}$ and $(\sqrt{\ln \ln (a^2 n p)}/(a^2 np) )^{2 \alpha / (4 \alpha +1)}$, respectively. However, the parameter $a$ describing the probability of a missing coordinate appears similarly in both cases. It actually deteriorates the rates with respect to the case $a=1$ of fully observed data. Proofs are given in section~\ref{sec:proofs}.
 
Note that, for the rest of the paper asymptotics will be taken when $n \to + \infty$, $p \to + \infty$ and $a$ is either fix or tends to 0 under further constraints.


\section{Test for covariance matrices} \label{sec:general}

 We want  to test from the data with missing coordinates $Y_1, \dots, Y_n$ the null hypothesis \eqref{H0} against the   alternative \eqref{H1} that we recall here:
\begin{equation*}
\label{H1simple}
  H_1 : \Sigma \in Q(\mathcal{F}(\alpha) , \varphi)
\end{equation*}
 where $Q(\mathcal{F}(\alpha) , \varphi)$ is given in \eqref{classe generale}.
This testing problem is treated in  \cite{ButuceaZgheib2014A}, for the case of fully observed data, which correspond to  $a=1$ in our case. For the sake of clarity, let us recall that in  \cite{ButuceaZgheib2014A}, the following weighted U-statistic was studied
$$
\widehat{\mathcal{D}}_{n,p} = \ds\frac{1}{n(n-1)p}   \ds\sum_{1 \leq k \ne l \leq n} \ds\sum_{ \substack{1 \leq i<j \leq p \\ |i-j| < m }} w_{ij} X_{k,i} X_{k,j} X_{l,i} X_{l,j} .
$$  
The test based on $\widehat{\mathcal{D}}_{n,p}$ was shown to achieve minimax and sharp minimax separation rates, i.e. asymptotic equivalents of the type II error and the total error probabilities are also given when $\varphi \asymp \tilde \varphi_\alpha$. The weights $\{w_{ij}\}_{1 \leq i < j \leq p}$ depend on the parameter $\alpha$ and are  chosen as solution of the following  optimization problem:
\begin{eqnarray}
\label{probextremal}
  \sup\limits_{ \{ w_{ij}\geq 0; \sum_{i<j}w_{ij}^2 = \frac 12 \}} \inf\limits_{\Sigma \in Q(\mathcal{F}(\alpha), \varphi)}  \mathbb{E}_\Sigma(\widehat{\mathcal{D}}_{n,p}) &=&   \sup\limits_{ \{ w_{ij}\geq 0; \sum_{i<j}w_{ij}^2 = \frac 12 \}}\inf\limits_{\Sigma \in Q(\mathcal{F}(\alpha), \varphi)}   \ds\frac 1p \sum_{i<j} w_{ij} \sigma^2_{ij} \nonumber \\
  & = & \inf\limits_{\Sigma \in Q(\mathcal{F}(\alpha), \varphi)}  \ds\frac 1p  \ds\sum_{i<j}\sigma_{ij}^4 =: b(\varphi),
\end{eqnarray}
where $b(\varphi) \sim C^{\frac 12}(\alpha) \varphi^{2 + \frac{1}{2 \alpha}}$ with $C(\alpha)= (2 \alpha +1)/((4\alpha +1)^{1 + \frac 1{2\alpha}})$, if $\varphi \to 0$ such that $p \varphi^{1/ \alpha} \to + \infty$.

In the next section we introduce a simpler U-statistic for the case of partially observed vectors and give the asymptotic minimax separation rates, then we aggregate these tests in order to construct a procedure free of the parameter $\alpha$.

\subsection{ Test procedure and minimax separation rate}
\label{sec : separation rate general case} 
Let us introduce the asymptotically minimax consistent test procedure  with simpler form than $\widehat{\mathcal{D}}_{n,p}$ defined above.
For an integer $ m \in \mathbb{N}$ large enough, such that it verifies
\begin{equation}
\label{propriete de m}
 D \leq m^{\alpha} \cdot \varphi  \leq K ^{-2 \alpha } \text{ for some constants } D>1 \,\,\,  \text{ and }  K >0,
\end{equation}
we define the following test statistic
\begin{equation}
\label{teststat}
\widehat{\mathcal{D}}_{n,p,m} = \ds\frac{1}{n(n-1)p} \cdot \frac{1}{ \ds\sqrt{2m} }  \ds\sum_{1 \leq k \ne l \leq n} \ds\sum_{ \substack{1 \leq i<j \leq p \\ |i-j| < m }}  Y_{k,i} Y_{k,j} Y_{l,i} Y_{l,j} \, .
\end{equation}
Note that, as in \citep{ButuceaZgheib2014A} we only use $m$ diagonals of the sample covariance matrix $\bar{Y} \bar{Y}^\top$, but the weights are constant and equal to $1/ \ds\sqrt{2m}$.
\begin{proposition}
\label{prop: statistic properties}
Under the null hypothesis, the test statistic $\widehat{\mathcal{D}}_{n,p,m} $  in \eqref{teststat} is a centered random variable with variance $\Var_{a,I}(\widehat{\mathcal{D}}_{n,p,m} ) = a^4 / (n(n-1)p)$.
Moreover, 
\[
\ds\frac{n \ds\sqrt{p}}{a^2} \cdot \widehat{\mathcal{D}}_{n,p, m}  \to \mathcal{N}(0,1) \quad  \text{ under $P_I$- probability}
\]
Under the alternative hypothesis, for all $\Sigma \in \mathcal{F}(\alpha)$ with $\alpha >1/2$,
\[
\mathbb{E}_{a,\Sigma}(\widehat{\mathcal{D}}_{n,p,T}) = \frac{a^4}{p \cdot \ds\sqrt{2m}} \ds\sum_{ \substack{1 \leq i<j \leq p \\ |i-j| \leq  m}}  \sigma_{ij}^2  \quad \text{and} \quad \Var_{a,\Sigma} = \ds\frac{T_1}{n(n-1)p^2}+ \frac{T_2}{np^2}
\]
where, for  $ m  \to + \infty$ such that $m/p \to 0$ and that \eqref{propriete de m} holds, 
\begin{eqnarray*}
T_1 & \leq &  p \cdot a^4 ( 1 + o(1)) +   p  \cdot \mathbb{E}_{a,\Sigma}(\widehat{\mathcal{D}}_{n, p,m})  \cdot O( a^2 m\ds\sqrt{m}) , \\
T_2 & \leq &  p  \cdot  \mathbb{E}_{a,\Sigma}(\widehat{\mathcal{D}}_{n, p,m}) \cdot O(a^2 \ds\sqrt{m})  +  p^{3/2} \Big(  \mathbb{E}^{3/2}_{a,\Sigma}(\widehat{\mathcal{D}}_{n, p,m})O(a^2 m^{3/4}) +   \mathbb{E}_{a,\Sigma}(\widehat{\mathcal{D}}_{n, p,m}) \cdot o(a^4) \Big) .
\end{eqnarray*}
\end{proposition}
\begin{proof}[Proof of Proposition \ref{prop: statistic properties}]
The proof follows the same steps as the proof of Proposition 1 of \cite{ButuceaZgheib2014A}. We use repeatedly the independence of  $( \varepsilon_{k,i})_{k,i}$ and $(X_{k,.})_k$ and obvious properties of the Bernoulli random variables.
\hfill \end{proof}

Now, we propose the following test procedure
\begin{equation}
\label{testm}
  \Delta_m = \Delta_m(t) = \mathds{1} ( \widehat{\mathcal{D}}_{n,p,m} > t) ,  \quad t>0
 \end{equation}
where $\widehat{\mathcal{D}}_{n,p,m}$ is the test statistic defined in \eqref{teststat}.
\begin{theorem}
\label{theo:upperbound1}
\noindent

Upper bound: let $m \to + \infty$ such that $ m/p \to 0$ and that \eqref{propriete de m} holds. If $\alpha > 1/2$ and if
\[
\varphi \to 0 \,  \text{ and } a^2 n \sqrt{p} \, \varphi^{2 + \frac 1{2 \alpha} } \to + \infty
\]
the test procedure defined in \eqref{testm} with $t \leq c \cdot a^4 \varphi^{2 + \frac 1{2 \alpha}}$ for some constant $c <  K(1 - D^{-2})/ \ds\sqrt{2}$ and with $n \ds\sqrt{p}\,t/a^2 \to + \infty$, is asymptotically minimax consistent, that is 
$$ \gamma( \Delta_m(t), Q(\mathcal{F}(\alpha), \varphi)) \to 0 .$$

Lower bound: if $ \alpha > 1/2$ and if
\[
a^2n \to + \infty \, , \quad p = o(1) \cdot (a^2 n)^{4 \alpha -1} \quad  \text{ and } \quad a^2 n \sqrt{p} \, \varphi^{2 + \frac 1{2 \alpha} } \to 0 
\]
then
\[
\gamma = \inf\limits_{\Delta} \gamma(\Delta, Q(\mathcal{F}( \alpha), \varphi) ) \to 1
\]
\end{theorem}
\begin{proof}[Proof of Theorem~\ref{theo:upperbound1}] The proof is given in section~\ref{sec:proofs}.
\end{proof}

From the previous theorem we deduce that the minimax separation rate is given by:
\[
\widetilde\varphi_\alpha = \Big( \ds {a^2  n \ds\sqrt{p}} \Big)^{-\frac{2 \alpha}{4 \alpha +1}}
\]
Thus the separation $\widetilde\varphi_\alpha$ obtained for the observations with missing values is slower by the $a^2$ factor than the separation rate obtained in the case of fully observed vectors. 

Note that the conditions on $t$, the threshold of $\Delta_m(t)$ in (\ref{testm}), are compatible. Indeed, $a^2 / (n \sqrt{p}) \ll c \cdot a^4 \varphi^{2 + 1/(2\alpha)}$ is equivalent to our assumption that $a^2 n \sqrt{p} \, \varphi^{2 +  1/(2 \alpha )} \to \infty$.

\subsection{Adaptation}
\label{sec: adaptivity general case}
In this section we construct an asymptotically adaptive consistent test procedure $\Delta_{ad}$ free of the parameter $\alpha \in A:=[\alpha_* , \alpha^*_{n,p} ] \subset ]1/2 , + \infty[$,  with $\alpha^*_{n,p} \to + \infty $ and $\alpha^*_{n,p} = o(1)\ln (a^2 n \ds\sqrt{p})$, to test $H_0$  given in \eqref{H0} against the large alternative 
\begin{equation*}
\label{H1adap1}
  H_1 : \Sigma \in \underset{\alpha \in A}{\bigcup} \Big\{ \mathcal{F}(\alpha) ~ ; ~
 \ds\frac 12  \ds \sum_{i<j} \sigma_{ij}^2  \geq (\mathcal{C} \Phi_\alpha)^2 \Big\} ,
\end{equation*}
where $\mathcal{C}>0$ is some constant and 
\[
\Phi_\alpha = \left(   \ds\frac{\ds\sqrt{\ln \ln (a^2n\ds\sqrt{p})}}{a^2n\ds\sqrt{p}} \right)^{ \frac{2\alpha}{4 \alpha +1}}.
\]
For each $\alpha \in [\alpha_* , \alpha^*_{n,p} ] $, there exists  $ l \in \mathbb{N}^*$ such that
\[
2^{l-1} \leq   (\Phi_\alpha)^{- \frac 1{\alpha}}  < 2^l, \quad  \text{it suffices to take } l \sim \ds\frac{\frac 2{4 \alpha +1 } \ln(a^2n \ds\sqrt{p})}{\ln (2)} .
\]
Let $L_*, L^* \in \mathbb{N}^*$  be defined by
\[  L_* = \Big( \ds\frac {2}{ (4 \alpha^*_{n,p} +1)\ln 2} \Big)  \ln (a^2n \ds\sqrt{p})  \quad \text{and} \quad
L^* = \Big( \frac 2{(4 \alpha_* +1) \ln 2 } \Big)\ln(a^2n \ds\sqrt{p}).
\]
We see that $L_* $ a,d $L^*$ tend to infinity. We define the adaptive  test procedure as follows
\begin{equation}
\label{testadap1}
\Delta_{ad}= \max_{L_* \leq l \leq L^*} \mathds{1}(\mathcal{D}_{n,p, 2^l} > t_l),
\end{equation}  
where $\mathcal{D}_{n,p, 2^l}$ is the test statistic defined in \eqref{teststat}, with $m$ replaced by $2^l$. 

\begin{theorem}
\label{theo : adaptivitygeneralcase}
The test procedure $\Delta_{ad}$ defined in \eqref{testadap1}, with $t_l = a^2\ds\frac{\ds\sqrt{\mathcal{C}^* \ln l }}{n\ds\sqrt{p}} $, verifies :

\noindent
Type I error probability : $\eta(\Delta_{ad}) \to 0$, for $ \mathcal{C}^* >4$.

\noindent
Type II error probability :   if 
\[
a^2n \ds\sqrt{p} \to + \infty \, ,2^{L^*}/p \to 0 \, , \, \ln( a^2 n \ds\sqrt{p})/n \to 0 ~ \text{ and } \mathcal{C}^2 > 1 + 4 \ds\sqrt{\mathcal{C}^*}
\]
we get 
\[
\beta(\Delta_{ad} \, , \underset{\alpha \in A}{\cup} Q( \mathcal{F}(\alpha), \mathcal{C} \Phi_\alpha) \to 0.
\]
\end{theorem}
Note that the condition $2^{L^*}/p \to 0$ is equivalent to  $a^2 n \ll p^{2 \alpha_*}$.

\begin{proof}[Proof of Theorem~\ref{theo : adaptivitygeneralcase}] The proof of this theorem is similar the the proof of the Theorem \ref{theo : adaptivityToepltz} which  is given in section~\ref{sec:proofs}.
\end{proof}



\section{Toeplitz covariance matrices}
\label{sec: Toeplitz matrices}
In this section we assume that the covariance matrix $\Sigma$ is Toeplitz. In this case, we are interested   to test \eqref{H0} against the following alternative
\begin{equation}
\label{H1simpleToeplitz}
H_1 : \Sigma \in Q(\mathcal{T}(\alpha), \phi) 
\end{equation}
where $Q(\mathcal{T}(\alpha), \phi) $ is defined in \eqref{classe toeplitz} for $\phi$ instead of $\varphi$.
This testing problem is treated in \cite{ButuceaZgheib15}, for the particular  case $a=1$, where a weighted U-statistic $\widehat{\mathcal{A}}_{n,p} $ of order 2  is used to construct an asymptotically consistent test procedure that achieve the sharp separation rates. Similarly to the previous setup, we construct an asymptotically consistent test procedure with constant weights. Recall that in \cite{ButuceaZgheib2014A} the weights are defined as solution of 
 the following optimization problem:
\begin{equation}
\label{probextremal2}
\sup\limits_{ \{ w_j \geq 0 ; \sum_j w_j^2 = \frac 12 \} } \inf\limits_{\Sigma \in Q(\mathcal{T}(\alpha), \phi) } \ds\sum_{j \geq 1}w_j \sigma_j^2 = \inf\limits_{\Sigma \in Q(\mathcal{T}(\alpha), \phi) }\sum_{j\geq 1} \sigma_j^4= C^{1/2}(\alpha)\phi^{2 + \frac 1{2 \alpha}} 
\end{equation}
Remark that the  optimization problems given in \eqref{probextremal} and \eqref{probextremal2} have the same solution when $\phi \to 0$ such that $p \phi^{1/\alpha} \to + \infty$.

\subsection{Test procedure and separation rates}
\label{sec: separation rate cas Toeplitz}
Take $m  \in \mathbb{N}$ such that $m \to + \infty$ and $m$ verifies \eqref{propriete de m} for $\phi$ instead of $\varphi$, 
we define the following test statistic:
\begin{equation}
\label{teststatToeplitz}
\widehat{\mathcal{A}}_{n, p, m}= \ds\frac{1}{n(n-1)(p-m)^2} \cdot \ds\frac 1{\ds\sqrt{2m}} \underset{1 \leq k \neq l \leq n}{\ds\sum }\sum_{j=1}^m  \underset{m+1 \leq i_1, i_2 \leq p}{\ds\sum } Y_{k, i_1} Y_{k, i_1-j} Y_{l, i_2} Y_{l, i_2-j}
\end{equation}
The main difference between the two test statistic $\widehat{\mathcal{D}}_{n, p, m} $ and $\widehat{\mathcal{A}}_{n, p, m}$ is that in this latter we take into consideration the fact that, we have  repeated information on the same diagonal elements.

Now, we give bounds on the moments of this test statistic :
\begin{proposition}
\label{prop : properties est toeplitz}
Under the null hypothesis $\widehat{\mathcal{A}}_{n, p, m}$ is a centered random variable whose variance is  $\Var_{a, I} (\widehat{\mathcal{A}}_{n, p, m})=  a^4/(n(n-1)(p-m)^2)$.
Moreover, we have that $ (n(p-m)/a^2) \cdot  \widehat{\mathcal{A}}_{n, p, m} / \to \mathcal{N}(0,1)$.
\noindent
Under the alternative hypothesis, for all $\Sigma \in \mathcal{T}(\alpha)$,
 \[
\mathbb{E}_{a, \Sigma}(\widehat{\mathcal{A}}_{n, p, m}) = (a^4/\ds\sqrt{2m}) \sum_{j =1}^m \sigma_j^2 \text{ and } 
\Var_{a, \Sigma}=  \ds\frac{R_1}{ n(n-1)(p-m)^4} + \ds\frac{R_2}{n(p-m)^2}, 
\]
where
\begin{eqnarray*}
R_1 &\leq & (p-m)^2 \cdot \{ a^4(1 + o(1)) +  \mathbb{E}_\Sigma(\widehat{\mathcal{A}}_{n, p, m}) \cdot (O(a^2 m) + O(a^3 m^{3/2 - 2 \alpha})) \\
& + & \mathbb{E}^2_\Sigma(\widehat{\mathcal{A}}_{n, p, m}) \cdot O(m^2 )  \}\\
R_2 &\leq & (p-m) \cdot \{a^2 \cdot \mathbb{E}_{\Sigma} (\widehat{\mathcal{A}}_{n, p, m}) \cdot o(1)  +  \mathbb{E}^{3/2}_{\Sigma}(\widehat{\mathcal{A}}_{n, p, m} ) \cdot (O( a\cdot m^{1/4})+ O(a^2 m^{3/4- \alpha})) \\
& + &   \mathbb{E}^2_{\Sigma}(\widehat{\mathcal{A}}_{n, p, m}) \cdot O(m)  \}.
\end{eqnarray*}
\end{proposition}
 It is easy to show that, since $m$ verifies \eqref{propriete de m}, we have  for all $\Sigma \in \mathcal{T}(\alpha)$
\[
\mathbb{E}_{a, \Sigma}(\widehat{\mathcal{A}}_{n, p, m}) \geq a^4 B \cdot 
\phi^{2+ \frac 1{2 \alpha}}
\] 
where $B$ is given in \eqref{minoration de l'espreance}.

To test \eqref{H0} against \eqref{H1simpleToeplitz}, we define the following test procedure based on the statistic defined in \eqref{teststatToeplitz} :
\begin{equation}
\label{testtoeplitz}
  \Delta_m^T = \Delta^T_m(t) = \mathds{1} ( \widehat{\mathcal{A}}_{n,p,m} > t) ,  \quad t>0.
 \end{equation}
 
\begin{theorem}
\label{theo:upperboundToeplitz}  
\noindent

Upper bound: let $m \to + \infty$, such that $ m/p \to 0 $ and that  \eqref{propriete de m} holds. If  $ \alpha > 1/4 $ and if 
\[
  \phi \to 0 \text{ and } a^2 n p \phi^{2 + \frac 1{2 \alpha}} \to + \infty
\]
the test procedure defined in \eqref{testm} with $t \leq \kappa \cdot a^4  \phi^{2 + \frac 1{2 \alpha}}$ for some constant $\kappa \leq B$ such that $n p \, t/a^2 \to + \infty$ is consistent, that is $ \gamma( \Delta^T_m(t), Q(\mathcal{T}(\alpha), \phi)) \to 0$.

Lower bound: if $\alpha > 1/2$ and if
\[
   a^2np \to +\infty \quad \text{and } a^2 n p \phi^{2 + \frac 1{2\alpha} }  \to 0
\]
then 
\[
\gamma = \inf\limits_{\Delta} \gamma (\Delta, Q(\mathcal{T}(\alpha), \phi) ) \to 1.
\]
\end{theorem}
 The main consequence of Theorem \ref{theo:upperboundToeplitz} is that the separation rate is given as follows :
\[
\widetilde\phi = \Big( \ds {a^2  n p} \Big)^{- \frac{2 \alpha}{4 \alpha +1}}
\]
\begin{proof}[Proof of Theorem \ref{theo:upperboundToeplitz}] The proof follows the same steps  as the proof of Theorem \ref{theo:upperbound1}, we therefore omit it. The most significant difference is that in order to show the lower bound in this case, we consider a sub-class of Toeplitz matrices:
\[
Q_T = \{ \Sigma_U : [\Sigma_U]_{ij} = \mathds{1}_{(i=j)}+u_{|i-j|}\sigma \mathds{1}_{(|1<|i - j|<T)} \text{ for all } 1 \leq i, j \leq p~,  ~U  \in \mathcal{U \}} ,
\]
where $\sigma$ and $T$ are defined in \eqref{sigma et T} and where
\[
\mathcal{U} = \{U= [u_{|i-j|}]_{1 \leq |i- j| \leq p-1 } :  u_{ii} =0 , \forall i \mbox{ and } \, u_{|i-j|}\pm 1 \cdot \mathds{1}_{(|i-j| < T)}, \mbox{ for } i\neq j \}.
\] 
Indeed, the signs are randomized but constant on each diagonal. We  re-write the terms of  $L_{n,p}$   taking into consideration the fact that the matrices are Toeplitz see, for example,  the proof of lower bound in \citep{ButuceaZgheib15}.
\end{proof}

\begin{remark}
Remark that the conditions on $m$ imply that $m$ is of order of $\phi^{- \frac 1{\alpha}}$ in the case of Toeplitz covariance matrices and of order of $\varphi^{- \frac 1{\alpha}}$ in the case of general covariance matrices. 
\end{remark} 



\subsection{Adaptation}
\label{sec: adaptivity Toeplitz}
In this section, it is always assumed that the covariance matrices are Toeplitz. Our goal is to construct a consistent test procedure independent of the parameter $\alpha \in A:=[\alpha_* , \alpha^*_{n,p}] \subset ]1/4, + \infty[$, such that $\alpha^*_{n,p} \to + \infty$ and $\alpha^*_{n,p} =o(1) \ln(a^2np)$, to test $H_0$  given in \eqref{H0} against the large alternative 
\begin{equation}
\label{H1comp}
  H_1 : \Sigma \in \underset{\alpha \in A}{\bigcup} \Big\{ \mathcal{T}(\alpha) ~ ; ~
   \ds \sum_{j =1}^{p-1} \sigma_j^2  \geq (\mathcal{C} \psi_\alpha)^2 \Big\} ,
\end{equation}
where $\mathcal{C}  >0$ is some constant and 
\[
\psi_\alpha = \left(   \ds\frac{\ds\sqrt{\ln \ln (a^2np )}}{a^2np} \right)^{ \frac{2\alpha}{4 \alpha +1}}.
\]
 First,  see that  $\forall \alpha \in[\alpha_* , \alpha^*_{n,p}]$, $\exists l \in \mathbb{N}^*$ such that
\[
2^{l-1} \leq  (\psi_\alpha)^{- \frac 1{\alpha}}  < 2^l, \quad  \text{it suffices to take } l \sim \ds\frac{\frac 2{4 \alpha +1 } \ln(a^2np)}{\ln(2)} 
\]
Let $L_*,\, L^* \in \mathbb{N}^*$ be defined by 
\[
L_* = \Big( \frac 2{(4 \alpha^*_{n,p} +1) \ln2 } \Big) \ln(a^2np) \quad \text{and} \quad L^* = \Big( \frac 2{(4 \alpha_* +1) \ln2 } \Big)\ln(a^2np)
\]
We aggregate tests for all given values of $l$ from $L_*$ to $L^*$ giving the following test procedure free of the parameter $\alpha$:
\begin{equation}
\label{testadap}
\Delta_{ad}= \max_{L_* \leq l \leq L^*} \mathds{1}(\mathcal{A}_{n,p, 2^l} > t_l),
\end{equation}  
where $\mathcal{A}_{n,p, 2^l}$ is the test statistic defined in \eqref{teststatToeplitz}, with $m$ replaced by $2^l$. 
\begin{theorem}
\label{theo : adaptivityToepltz}
The test procedure $\Delta_{ad}$ defined in \eqref{testadap}, with $t_l = a^2\ds\frac{\ds\sqrt{\mathcal{C}^* \ln l }}{n(p-2^l)} $, verifies :

\noindent
Type I error probability : $\eta(\Delta_{ad}) \to 0$, for $ \mathcal{C}^* >4$.

\noindent
Type II error probability :   if 
\[
a^2n p \to + \infty \, , 2^{L^*}/p \to 0 \, , \, ln( a^2 np)/n \to 0 ~ \text{ and } \mathcal{C}^2 \geq 1 + 4 \ds\sqrt{\mathcal{C}^*}
\]
we get 
\[
\beta(\Delta_{ad}, \underset{\alpha \in A}{\cup} Q( \mathcal{T}(\alpha), \mathcal{C} \psi_\alpha) \to 0.
\]
\end{theorem}
\begin{proof}[Proof of Theorem~\ref{theo : adaptivityToepltz}] The proof is given in section~\ref{sec:proofs}.
\end{proof}

Remark that the condition $2^{L^*}/p $ gives that $ a^2n \ll p^{2 \alpha_* - \frac 12}$ and $\ln(a^2np)/n \to 0$ implies that $a^2np \ll e^n$. Together, these conditions are mild as they give $a^2np \ll \min \{ p^{2 \alpha_* + \frac 12} , e^n \}$.

\section{Proofs} \label{sec:proofs}

\begin{proof}[Proof of Theorem \ref{theo:upperbound1}]
{ \bf  Upper bound }
We use the asymptotic normality of $\widehat{\mathcal{D}}_{n,p,m} $ to show that, the type I error probability 
\[
\eta( \Delta_m(t)) = P_{a, I} ( \Delta_m(t) = 1) =  P_{a, I} (\widehat{\mathcal{D}}_{n,p,m} > t) = \Phi\Big( - \frac{ n \ds\sqrt{p} t}{a^2}\Big) + o(1) = o(1)
\] 
as soon as $  n \ds\sqrt{p} \,t/ a^2 \to + \infty$.
In order to control the maximal type II error probability,  we use the Markov inequality to get that for all $\Sigma$ in $Q(\mathcal{F}(\alpha), \varphi)$: 
\begin{eqnarray}
 P_{a, \Sigma} ( \Delta_m(t) = 0) & =&  P_{a, \Sigma}(\widehat{\mathcal{D}}_{n,p,m} < t) \leq P_{a, \Sigma} \Big(|\widehat{\mathcal{D}}_{n,p,m} - \mathbb{E}_{a, \Sigma}(\widehat{\mathcal{D}}_{n,p,m})| < \mathbb{E}_{a, \Sigma}(\widehat{\mathcal{D}}_{n,p,m})-t \Big) \nonumber \\
 & \leq & \ds\frac{\Var_{a, \Sigma}(\widehat{\mathcal{D}}_{n,p,m})}{ (\mathbb{E}_{a, \Sigma}(\widehat{\mathcal{D}}_{n,p,m})-t)^2}  =  \ds\frac{T_1 +(n-1)T_2}{n(n-1)p^2(\mathbb{E}_{a, \Sigma}(\widehat{\mathcal{D}}_{n,p,m})-t)^2}, \nonumber
\end{eqnarray}
for $t$ properly chosen.
In order to bound the previous quantity uniformly in $\Sigma$ over  $Q(\mathcal{F}(\alpha), \varphi)$ we need to control 
$$
\inf\limits_{\Sigma \in Q(\mathcal{F}(\alpha), \varphi )} \mathbb{E}_{a,\Sigma}(\widehat{\mathcal{D}}_{n,p,m}) =  \inf\limits_{\Sigma \in Q(\mathcal{F}(\alpha), \varphi)}  \frac{a^4}{p \cdot \ds\sqrt{2m}} \ds\sum_{ \substack{1 \leq i<j \leq p \\ |i-j|  <  m}}  \sigma_{ij}^2.
$$
For all $ \Sigma \in Q(\mathcal{F}(\alpha), \varphi)$, we have
\begin{eqnarray}
\mathbb{E}_{a,\Sigma}(\widehat{\mathcal{D}}_{n,p,m}) &=&  \frac{a^4}{p \cdot \ds\sqrt{2m}} \ds\sum_{ \substack{1 \leq i<j \leq p \\ |i-j| < m }}  \sigma_{ij}^2 =  \frac{a^4}{p \cdot \ds\sqrt{2m}} \Big( \ds\sum_{ 1 \leq i<j \leq p }  \sigma_{ij}^2 -  \ds\sum_{ \substack{1 \leq i<j \leq p \\ |i-j| \geq m }}  \sigma_{ij}^2 \Big) \nonumber \\
& \geq & \frac{a^4 }{ \ds\sqrt{2m}} \Big( \varphi^2 - \frac 1p \sum_{i<j} \ds\frac{|i-j|^{2\alpha}}{m^{2 \alpha}}  \sigma_{ij}^2 \Big) \geq \frac{a^4 \varphi^2}{  \ds\sqrt{2m}} \Big( 1 - \ds\frac 1{m^{2\alpha}  \varphi^2} \Big). \nonumber 
\end{eqnarray}
We use \eqref{propriete de m} to get that, for all $ \Sigma \in Q(\mathcal{F}(\alpha), \varphi)$
\begin{equation}
\label{minoration de l'espreance}
\mathbb{E}_{a,\Sigma}(\widehat{\mathcal{D}}_{n,p,m}) \geq   \ds\frac{a^4 K}{\ds\sqrt{2}} \cdot \varphi^{2 + \frac 1{2\alpha}} \Big(1 - \frac 1{D^2} \Big)=: a^4 B \cdot\varphi^{2 + \frac 1{2\alpha}} , \text{ where } B = \ds\frac{K}{\ds\sqrt{2}} ( 1 - D^{-2}).
\end{equation}
Therefore, take $t \leq c \cdot a^4 \varphi^{2+\frac 1{2\alpha}}$ for $c < B$ and use \eqref{minoration de l'espreance}   to obtain that
\begin{eqnarray}
\ds\frac{T_1}{n(n-1)p^2(\mathbb{E}_{a, \Sigma}(\widehat{\mathcal{D}}_{n,p,m})-t)^2} & \leq & \ds\frac{1 +o(1)}{a^4n(n-1)p \varphi^{4 + \frac 1\alpha }(B-c)^2} + \ds\frac{a^2 \cdot  O(m\ds\sqrt{m})}{a^4n(n-1)p\varphi^{2 + \frac 1{2\alpha}}(1-c/B)^2} \nonumber \\
 &=& o(1) \, ,\nonumber
\end{eqnarray}
 if $ a^4n(n-1)p \varphi^{4+1/\alpha} \to + \infty $ and for all $\alpha > 1/2$. Indeed, $ a^2 \cdot m \ds\sqrt{m} \varphi^{2 + 1/(2 \alpha)} \asymp a^2 \varphi^{2 - \frac 1{\alpha}} = o(1)$.
 Similarly we show that under the previous conditions the term $ T_2/np^2(\mathbb{E}_{a, \Sigma}(\widehat{\mathcal{D}}_{n,p,m})-t)^2$ tends to 0.
 
{ \bf Lower bound }
To show the lower bound we first restrict the class $Q(\mathcal{F}(\alpha), L)$ to the class 
\[
Q := \{ \Sigma_U : [\Sigma_U]_{ij} = \mathds{1}_{(i=j)}+u_{ij}\sigma \mathds{1}_{(|1<|i - j|<T)} \text{ for all } 1 \leq i, j \leq p~,  ~U  \in \mathcal{U \}} ,
\]
where
\begin{eqnarray}
\label{sigma et T}
 \sigma = \varphi^{1 + \frac{ 1}{2 \alpha}} \, ,  \quad T \asymp \varphi^{- \frac{1}{\alpha}},  
\end{eqnarray}
and
\begin{equation*}
\mathcal{U} = \{U= [u_{ij}]_{1 \leq i ,\, j\leq p } :  u_{ii} =0 , \forall i \mbox{ and } \, u_{ij} = u_{j\, i}= \pm 1 \cdot \mathds{1}_{(|i-j| < T)}, \mbox{ for } i\neq j \},
\end{equation*}
Denote  by  
$\varepsilon_{k}=(\varepsilon_{k,1}, \ldots, \varepsilon_{k,p})^\top$ the  random vector with i.i.d. entries $\varepsilon_{k,i} \sim \mathcal{B}(a)$, for all $1 \leq k \leq n$. Moreover 
denote by $P_{\varepsilon}$ and by $P_{\varepsilon_k}$ the distributions  of $ \varepsilon=(\varepsilon_{1}, \ldots,\varepsilon_{n}) $ and of $\varepsilon_{k} $, respectively.
Recall that the observations $Y_1, \cdots, Y_n$ verify $Y_k = \varepsilon_k * X_k$  for all $1 \leq k \leq n $,  where $*$ designate the Schur product. 

 Under the null hypothesis $X_1,\ldots, X_n \overset{i.i.d.}{\sim} \mathcal{N}(0, I)$, thus  the conditional random vectors $Y_k| \varepsilon_{k}$,  are independent Gaussian vectors such that,  for all  $1 \leq k \leq n$,  ~$Y_k| \varepsilon_{k} \sim \mathcal{N}(0, I* (\varepsilon_k \varepsilon_k^\top))$.
We denote respectively by $P_{ I}$  and by $P_{I}^{(\varepsilon)}$ the distributions of   $(Y_1, \ldots, Y_n)$ and of $(Y_1, \ldots, Y_n)|(\varepsilon_{1}, \ldots, \varepsilon_{n})$ under the null hypothesis.
Under the alternative hypothesis, for $X_1,\ldots, X_n \sim \mathcal{N}(0, \Sigma_U)$, we get that the conditional random vectors $Y_k| \varepsilon_{k}$ are independent Gaussian vectors such that $Y_k| \varepsilon_{k} \sim \mathcal{N}(0, \Sigma_{U} * (\varepsilon_k \varepsilon_k^\top) )$ for all $1 \leq k \leq n$, where
\[
(\Sigma_{U} * (\varepsilon_k  \varepsilon_k^\top))_{ij}=
\begin{cases}
\varepsilon_{ik}  \quad \text{ for } i=j \\
\varepsilon_{ik} \varepsilon_{jk} \cdot \sigma \quad \text{if } 1 < |i - j | <T \\
0  \quad \text{otherwise}.
\end{cases}
\]
We denote by $P_{U}=P_{\Sigma_U}$ and  $P_{U}^{(\varepsilon)}= P_{\Sigma_U}^{(\varepsilon)}$  the  distributions of   $(Y_1, \ldots, Y_n)$ and of 
the conditional distribution $(Y_1, \ldots, Y_n)|(\varepsilon_{1}, \ldots, \varepsilon_{n})$ respectively, when $X_1,\ldots, X_n \sim \mathcal{N}(0, \Sigma_{U})$.

We define the average distribution over $Q$ by
\[
P_{ \pi} = \ds\frac{1}{2^{p(T-1)/2}} \sum_{ U \in \mathcal{U}} P_{ U} .
\]
It is known (see \cite{IngsterSuslina03}) that the minimax total error  probability  satisfies
\[
\begin{array}{lcl}
\gamma \geq  1 - \ds\frac{1}{2}
\| P_I - P_{\pi}\|_1 
\end{array}
\]
In order to prove that $\gamma \longrightarrow 1 $, we bound from above the $L_1$ distance by the Kullback-Leibler  divergence (see \cite{Tsybakov09})
\[
\| P_I - P_{\pi}\|_1^2 \leq  \frac 12 \cdot  K(P_{ I}, P_{\pi}), \quad \mbox{where } K(P_{ I}, P_{\pi}) := \mathbb{E}_{I} \log \Big( \ds\frac{dP_{I}}{dP_{ \pi}} \Big) .
\]
Therefore, to complete the proof, it is sufficient to show that $K(P_{ I}, P_{\pi}) \to 0$.
In order to prove this we use the conditional likelihoods as follows:
\begin{eqnarray}
K(P_{I}, P_{ \pi}) 
&=& 
 \mathbb{E}_{I} \log \Big( \ds\frac{d(P_\varepsilon \otimes P_{I}^{(\varepsilon)})}{d(P_\varepsilon \otimes P^{(\varepsilon)}_{\pi})} \Big) 
=  \mathbb{E}_{\varepsilon}  \mathbb{E}_{I}^{(\varepsilon)}  \log \Big( \ds\frac{dP_{I}^{(\varepsilon)}}{dP_{\pi}^{(\varepsilon)}} \Big). \nonumber
\end{eqnarray}
Let $\varepsilon(w)$ be a realization of $\varepsilon$, we denote by 
 $S_k \in \{1 , \ldots, p \}$ the support of $\varepsilon_k(w)$  that is  $\varepsilon_{k,i}(w)=1 $ if and only if $i \in S_k$. Also we denote by   $d_k = \Card(S_k)$, $\Sigma_{U}^{ \varepsilon_k(w)}$ the positive matrix $\in \mathbb{R}^{d_k \times d_k}$, defined as the sub-matrix of $\Sigma_{U}$ obtained by removing all the $i$-th rows and columns corresponding to $i \notin S_k$ and $X_{\varepsilon_k(w)}$ the sub vector of $X_k$ of dimension $d_k$ in which we retain the coordinate with indices in $S_k$. Thus,   
\begin{eqnarray}
&& L((Y_1, \ldots, Y_n )| \varepsilon(w) ) := \mathbb{E}_{I}^{(\varepsilon(w))}  \log \Big( \ds\frac{dP_{I}^{(\varepsilon(w))}}{dP_{\pi}^{( \varepsilon(w))}} (Y_1, \ldots, Y_n )\Big)  \nonumber \\
&=& \mathbb{E}_{I}^{( \varepsilon(w))}  \left( - \log \mathbb{E}_U \exp \Big( 
-  \frac 12 \sum_{k=1}^n \Big( X_{\varepsilon_k(w)}^{\top} ((\Sigma_{U}^{\varepsilon_k(w)})^{-1}-I_{\varepsilon_k}) X_{\varepsilon_k(w)} + \log \det(\Sigma_{U}^{\varepsilon_k(w)})\Big) \Big)  \right) \nonumber
\end{eqnarray}
Therefore we have
\begin{eqnarray}
& & \mathbb{E}_\varepsilon \Big( L((Y_1, \ldots, Y_n )| \varepsilon) \Big) \nonumber \\
&=& \mathbb{E}_\varepsilon \mathbb{E}_{I}^{( \varepsilon)}  \left( - \log \mathbb{E}_U \exp \Big( 
-  \frac 12 \sum_{k=1}^n \Big( X_{\varepsilon_k}^{\top} ((\Sigma_{U}^{\varepsilon_k})^{-1}-I_{\varepsilon_k}) X_{\varepsilon_k} + \log \det(\Sigma_{U}^{\varepsilon_k})\Big) \Big)  \right) \nonumber
\end{eqnarray}
Denote by 
\begin{eqnarray}
L_{n,p } 
&:=&   \log \mathbb{E}_U \exp \Big( 
-  \frac 12 \sum_{k=1}^n \Big( X_{\varepsilon_k}^{\top} ((\Sigma_{U}^{\varepsilon_k})^{-1}-I_{\kappa_k}) X_{\varepsilon_k} + \log \det(\Sigma_{U}^{\varepsilon_k})\Big) \Big)    \label{L0}
\end{eqnarray}
We define  $\Delta_{U}^{\varepsilon_k} = \Sigma_{U}^{\varepsilon_k} - I^{\varepsilon_k}$, for all $U \in \mathcal{U}$ and any realization of $\varepsilon_k$, we have $tr (\Delta_{U}^{\varepsilon_k})=0 $ and  $\| \Delta_{U}^{\varepsilon_k} \| = O(\varphi^{1 - \frac 1{2 \alpha}})$ which is $ o(1)$, as soon as $\varphi \to 0$ and $ \alpha > 1/2 $. In fact, by the Gershgorin's theorem we have
\[
\| \Delta_{U}^{\varepsilon_k} \|  \leq \max_{i \in S_k} \sum_{\substack{j \in S_k \\j \ne i}}  |(\Delta_{U}^{\varepsilon_k})_{ij}| \leq  \max_{i \in S_k} \sum_{\substack{j \in S_k \\ 1< |i-j| <T}}  |u_{ij} \sigma| \leq  2 T \cdot \sigma = O(\varphi^{1 - \frac 1{2\alpha}}).
\]
For all $x \in [- \frac 12 , + \frac 12]$ we have the following inequalities
\[
\begin{array}{lclcl}
x-x^2+x^3-2x^4  & \leq & - \Big(\ds\frac 1{1+x} -1 \Big)  & \leq & x- x^2 +x^3\nonumber \\
-x + \ds\frac{x^2}2 - \frac{x^3}{3} & \leq &  - \log(1+x) & \leq & -x + \ds\frac{x^2}{2} - \frac{x^3}{3} + \frac{x^4}{2}.
\end{array}
\]
Applying these inequalities to the eigenvalues of $\Delta_{U}^{\varepsilon_k} $ we get
\begin{eqnarray}
\Delta_{U}^{\varepsilon_k} - ( \Delta_{U}^{\varepsilon_k} )^2  +  ( \Delta_{U}^{\varepsilon_k} )^3  - 2  ( \Delta_{U}^{\varepsilon_k} )^4 \leq  - ( (\Sigma_{U}^{\varepsilon_k})^{-1} - I^{\varepsilon_k})   & \leq &  \Delta_{U}^{\varepsilon_k} - ( \Delta_{U}^{\varepsilon_k} )^2  +  ( \Delta_{U}^{\varepsilon_k} )^3   \nonumber 
 \label{taylor1} 
\end{eqnarray}
\begin{eqnarray}
  \ds\frac{1}{2}tr (\Delta_{U}^{\varepsilon_k})^2  - \ds\frac{1}{3}  tr (\Delta_{U}^{\varepsilon_k} )^3  \leq - \log \det(\Sigma_{U}^{\varepsilon_k}) \leq   \ds\frac{1}{2}tr (\Delta_{U}^{\varepsilon_k})^2  - \ds\frac{1}{3}  tr (\Delta_{U}^{\varepsilon_k} )^3   + \frac 12 tr (\Delta_{U}^{\varepsilon_k} )^4 \, ,
 \nonumber \label{taylor2}
\end{eqnarray}
for $\varphi$ small enough such that $\|  \Delta_U^{\varepsilon_k}\| \leq 1/2 $.
Thus we can bound $L_{n,p}$,  $\underline{L}_{n,p} \leq  L_{n,p} \leq \bar{L}_{n,p}$ where 
\begin{eqnarray}
\underline{L}_{n,p} &:= &  \log \mathbb{E}_U \exp \Big( \ds\frac 12 \sum_{k=1}^n X_{\varepsilon_k}^{\top} (\Delta_{U}^{\varepsilon_k}  - ( \Delta_{U}^{\varepsilon_k} )^2  +  ( \Delta_{U}^{\varepsilon_k} )^3  -2  ( \Delta_{U}^{\varepsilon_k} )^4 ) X_{\varepsilon_k}  \nonumber \\
&&  \hspace{2 cm}
+  \ds\frac{1}{2} \sum_{k=1}^n \Big( \frac 12 tr (\Delta_{U}^{\varepsilon_k})^2  - \ds\frac{1}{3} tr (\Delta_{U}^{\varepsilon_k} )^3 \Big) \Big) \, , \text{ and } \nonumber \\
\bar{L}_{n,p} &:=& \log \mathbb{E}_U \exp \Big( \ds\frac 12 \sum_{k=1}^n X_{\varepsilon_k}^{\top} (\Delta_{U}^{\varepsilon_k} - ( \Delta_{U}^{\varepsilon_k} )^2 + ( \Delta_{U}^{\varepsilon_k} )^3 ) X_{\varepsilon_k} \nonumber \\
&&   \hspace{2 cm}  + \ds\frac{1}{2} \sum_{k=1}^n  \Big( \frac 12 tr (\Delta_{U}^{\varepsilon_k})^2  - \ds\frac{1}{3}  tr (\Delta_{U}^{\varepsilon_k} )^3   +  \frac 12  tr (\Delta_{U}^{\varepsilon_k} )^4  \Big) \Big). \nonumber 
\end{eqnarray}
Now we develop the terms of  $ \bar{L}_{n,p}$ 
\begin{eqnarray*}
tr (\Delta_{U}^{\varepsilon_k})^2 = 2 \sigma^2 \ds\sum_{\substack{i<j \\ 1 < |i-j| < T}} \varepsilon_{k,i} \varepsilon_{k,j},  \quad  \quad  tr (\Delta_{U}^{\varepsilon_k})^3 =  3! \sigma^3 \hspace{- 0.5 cm}\sum_{ \substack{ i<i_1 < i_2  \\ 1 < |i-i_1|, |i_1-i_2|, |i_2 -i|<T}}  \hspace{- 0.5 cm} u_{ii_1}u_{i_1 i_2} u_{i_2 i}    \varepsilon_{k,i} \varepsilon_{k,i_1} \varepsilon_{k,i_2}.
\end{eqnarray*}
and 
\begin{eqnarray*}
tr(\Delta_{U}^{\varepsilon_k})^4 &=& \sigma^4 \sum_{ \substack{i, i_1, i_2, i_3  \\ 1 < |i-i_1|, \ldots, |i_3 -i|<T}} u_{ii_1}u_{i_1 i_2} u_{i_2 i_3} u_{i_3 i} \varepsilon_{k,i} \varepsilon_{k, i_1}  \varepsilon_{k, i_2} \varepsilon_{k,i_3} \nonumber \\
&=& \sigma^4 \sum_{  1 < |i-i_1|< T}  \varepsilon_{k,i} \varepsilon_{k, i_1} +  2 \sigma^4 \sum_{ \substack{i, i_1, i_2 \\ 1 < |i-i_1|, |i_1 -i_2|<T}} \varepsilon_{k,i} \varepsilon_{k, i_1}  \varepsilon_{k, i_2}  \nonumber \\
&+& 4! \,\, \sigma^4 \sum_{ \substack{i < i_1 < i_2 < i_3  \\ 1 < |i-i_1|, \ldots, |i_3 -i|<T}} u_{ii_1}u_{i_1 i_2} u_{i_2 i_3} u_{i_3 i} \varepsilon_{k,i} \varepsilon_{k, i_1}  \varepsilon_{k, i_2} \varepsilon_{k,i_3} 
\end{eqnarray*}
Moreover, we have (recall that $u_{ij}^2 =1$ and $\varepsilon_{ij}^2 = \varepsilon_{ij}$)
\begin{eqnarray}
\sum_{k=1}^n X_{\varepsilon_k}^\top  \Delta_{U}^{\varepsilon_k} X_{\varepsilon_k} &=& 2 \sigma \cdot  \sum_{\substack{ i < j \\ 1 < |i-j| <T}}  u_{ij}  \sum_{k=1}^n \varepsilon_{k,i} \varepsilon_{k,j} X_{k,i} X_{k,j} \nonumber \\
\sum_{k=1}^n X_{\varepsilon_k}^\top  (\Delta_{U}^{\varepsilon_k})^2 X_{\varepsilon_k} &=& \sigma^2    \ds\sum_{k=1}^n \sum_{i,j}   X_{\varepsilon_{k,i}}  X_{\varepsilon_{k,j}} \sum_{ \substack{ i_1 \\ 1 < |i_1 - i|, |i_1 -j| < T}} u_{i i_1 } u_{i_1 j} \varepsilon_{k,i} \varepsilon_{k,i_1} \varepsilon_{k,j} \nonumber \\ 
&=& \sigma^2    \ds\sum_{k=1}^n \sum_{i}   X_{\varepsilon_{k,i}}^2  \sum_{ \substack{ i_1 \\ 1 < |i_1 - i|< T}} \varepsilon_{k,i} \varepsilon_{k,i_1}  \nonumber \\
&+& 2 \sigma^2    \ds\sum_{k=1}^n \sum_{i < j}   X_{\varepsilon_{k,i}}  X_{\varepsilon_{k,j}} \sum_{ \substack{ i_1 \\ 1 < |i_1 - i|, |i_1 -j| < T}} u_{i i_1 } u_{i_1 j} \varepsilon_{k,i} \varepsilon_{k,i_1} \varepsilon_{k,j} \, , \nonumber 
\end{eqnarray}
and
\begin{eqnarray}
 \ds\sum_{k=1}^n    X_{\varepsilon_k}^{\top} (\Delta_{U}^{\varepsilon_k} )^3  X_{\varepsilon_k} 
& =& \sigma^3   \ds\sum_{k=1}^n \sum_{i,j}   X_{\varepsilon_{k,i}}  X_{\varepsilon_{k,j}} \sum_{ \substack{i_1 , i_2 \\ 1 < |i-i_1|, |i_1-i_2|, |i_2 -j|<T}} u_{ii_1}u_{i_1 i_2} u_{i_2 j}  \varepsilon_{k,i} \varepsilon_{k,i_1} \varepsilon_{k,i_2} \varepsilon_{k,j} \nonumber  \\
&=& 2  \sigma^3  \ds\sum_{k=1}^n \sum_{i } X_{\varepsilon_{k,i}}^2 \sum_{\substack{ i_1 < i_2 \\  1 < |i-i_1|, |i_1-i_2|, |i_2 -i|<T}} u_{ii_1}u_{i_1 i_2} u_{i_2 i}  \varepsilon_{k,i} \varepsilon_{k,i_1} \varepsilon_{k,i_2} \nonumber \\
&+& 2\sigma^3  \ds\sum_{k=1}^n \sum_{ \substack{ i < j \\ 1< |i-j| < T }}  X_{\varepsilon_{k,i}}  X_{\varepsilon_{k,j}} \Big(u_{ij}^3 \varepsilon_{k,i}  \varepsilon_{k,j} +  2 \sum_{\substack{ i_1 \\ 1 < |i_1 -i | <T} }   u_{ij}  \varepsilon_{k,i} \varepsilon_{k,i_1} \varepsilon_{k,j}  \Big)  \nonumber \\
&+&   2\sigma^3  \ds\sum_{k=1}^n \sum_{ i < j}  X_{\varepsilon_{k,i}}  X_{\varepsilon_{k,j}} \sum_{ \substack{ j \ne i_1 \ne i_2 \ne i \\ 1 < |i-i_1|, |i_1-i_2|, |i_2 -j|<T}} u_{ii_1}u_{i_1 i_2} u_{i_2 j}  \varepsilon_{k,i} \varepsilon_{k,i_1} \varepsilon_{k,i_2} \varepsilon_{k,j} \,  \nonumber 
\end{eqnarray}
In consequence, $\bar{L}_{n,p}$ can be written as follows:
\begin{eqnarray}
\bar{L}_{n,p} &=& \log \mathbb{E}_U \exp \Big\{ \sum_{\substack{ i < j \\ 1 < |i-j| <T}}  u_{ij}  \sum_{k=1}^n \varepsilon_{k,i} \varepsilon_{k,j} X_{k,i} X_{k,j}   \Big(  \sigma +  \sigma^3(1 + 2\sum_{\substack{ i_1 \\ 1 < |i_1 -i | <T} }    \varepsilon_{k,i_1}) \Big) \nonumber \\\nonumber \\
&- &    \sum_{ \substack{ i,i_1,j \\ i<j \\ 1 < |i_1 - i|, |i_1 -j| < T}} u_{i i_1 } u_{i_1 j} \ds\sum_{k=1}^n   X_{\varepsilon_{k,i}}  X_{\varepsilon_{k,j}} \varepsilon_{k,i} \varepsilon_{k,i_1} \varepsilon_{k,j} \, \sigma^2  \nonumber \\
&+&      \sum_{\substack{i< i_1 < i_2  \\  1 < |i-i_1|, |i_1-i_2|, |i_2 -i|<T}} u_{ii_1}u_{i_1 i_2} u_{i_2 i} \ds\sum_{k=1}^n \varepsilon_{k,i} \varepsilon_{k,i_1} \varepsilon_{k,i_2} \sigma^3 \Big( 3  X_{\varepsilon_{k,i}}^2  - 1 \Big) \nonumber \\
&+&   \sum_{ \substack{i,,i_1, i_2,j \\ i<j \\ j \ne i_1 \ne i_2 \ne i \\ 1 < |i-i_1|, |i_1-i_2|, |i_2 -j|<T}} u_{ii_1}u_{i_1 i_2} u_{i_2 j}  \ds\sum_{k=1}^n   X_{\varepsilon_{k,i}}  X_{\varepsilon_{k,j}} \varepsilon_{k,i} \varepsilon_{k,i_1} \varepsilon_{k,i_2} \varepsilon_{k,j} \, \sigma^3 \nonumber \\
&+& 6  \sum_{ \substack{i < i_1 < i_2 < i_3  \\ 1 < |i-i_1|, \ldots, |i_3 -i|<T}} u_{ii_1}u_{i_1 i_2} u_{i_2 i_3} u_{i_3 i} \sum_{k=1}^n \varepsilon_{k,i} \varepsilon_{k, i_1}  \varepsilon_{k, i_2} \varepsilon_{k,i_3} \, \sigma^4
 \, \Big\} \nonumber \\
& +& \ds\frac{\sigma^2}4 \ds\sum_{\substack{i<j \\ 1 < |i-j| < T}} \sum_{k=1}^n \varepsilon_{k,i} \varepsilon_{k,j} 
 -   \ds\frac{\sigma^2}2  \sum_{ \substack{i < i_1 \\ 1 < |i_1 - i|< T}}  \ds\sum_{k=1}^n   X_{\varepsilon_{k,i}}^2\varepsilon_{k,i} \varepsilon_{k,i_1} \nonumber \\
&+& \frac{\sigma^4}4 \sum_{ \substack{i, i_1 \\ 1 < |i-i_1|< T}}  \sum_{k=1}^n \varepsilon_{k,i} \varepsilon_{k, i_1}
+ \frac{\sigma^4}2  \sum_{ \substack{i, i_1, i_2 \\ 1 < |i-i_1|, \ldots, |i_2 -i|<T}} \sum_{k=1}^n \varepsilon_{k,i} \varepsilon_{k, i_1}  \varepsilon_{k, i_2}  \nonumber 
\end{eqnarray}
We have that $\{ u_{ij} \}_{\substack{1 \leq i < j \leq p}}$ is a sequence of i.i.d Rademacher random variables. Note that sequences composed of finite products of i.i.d Rademacher random variables, for example the sequences  $\{ u_{ir} u_{rj} \}_{\substack{1 < i \ne r \ne j  \leq p \\ i<j }} $ and $ \{ u_{ir} u_{rs} u_{sj} \}_{ \substack{1 \leq i \ne r \ne s \ne j \leq p \\ i <j }}$, form sequences of i.i.d Rademacher random variables. Moreover they are mutually  independent  and 	 independent from the initial sequence $\{ u_{ij} \}_{i<j}$. 
Now we explicit in $\bar{L}_{n,p}$ the expected value with respect to the i.i.d Rademacher  random variables and get

\begin{eqnarray}
\bar{L}_{n,p} &=&  \sum_{\substack{ i < j \\ 1 < |i-j| <T}}   \log \cosh \Big( \sum_{k=1}^n \varepsilon_{k,i} \varepsilon_{k,j} X_{k,i} X_{k,j} \Big(  \sigma +  \sigma^3(1 + 2\sum_{\substack{ i_1 \\ 1 < |i_1 -i | <T} }    \varepsilon_{k,i_1}) \Big) \Big) \nonumber \\
& +& \ds\frac{\sigma^2}2 \ds\sum_{\substack{i<j \\ 1 < |i-j| < T}} \sum_{k=1}^n \varepsilon_{k,i} \varepsilon_{k,j} 
- \sigma^2 \sum_{ \substack{i < i_1 \\ 1 < |i_1 - i|< T}}  \ds\sum_{k=1}^n   X_{\varepsilon_{k,i}}^2\varepsilon_{k,i} \varepsilon_{k,i_1} \nonumber \\
&+&  \sum_{ \substack{ i,i_1,j \\ i<j \\ 1 < |i_1 - i|, |i_1 -j| < T}}   \log \cosh \Big(  \ds\sum_{k=1}^n   X_{\varepsilon_{k,i}}  X_{\varepsilon_{k,j}} \varepsilon_{k,i} \varepsilon_{k,i_1} \varepsilon_{k,j} \, \sigma^2 \Big) \nonumber \\
&+&   \sum_{\substack{i< i_1 < i_2  \\  1 < |i-i_1|, |i_1-i_2|, |i_2 -i|<T}}  \log \cosh \Big( 3 \sigma^3 \ds\sum_{k=1}^n X_{\varepsilon_{k,i}}^2 \varepsilon_{k,i} \varepsilon_{k,i_1} \varepsilon_{k,i_2} -  \sigma^3  \Big)   \nonumber \\
&+&   \sum_{ \substack{i,i_1, i_2,j \\ i<j \\ j \ne i_1 \ne i_2 \ne i \\ 1 < |i-i_1|, |i_1-i_2|, |i_2 -j|<T}} \log \cosh \Big(   \ds\sum_{k=1}^n   X_{\varepsilon_{k,i}}  X_{\varepsilon_{k,j}} \varepsilon_{k,i} \varepsilon_{k,i_1} \varepsilon_{k,i_2} \varepsilon_{k,j} \, \sigma^3 \Big) \nonumber \\ 
&+&  \sum_{ \substack{i < i_1 < i_2 < i_3  \\ 1 < |i-i_1|, \ldots, |i_3 -i|<T}} \log \cosh \Big( 6 \sum_{k=1}^n \varepsilon_{k,i} \varepsilon_{k, i_1}  \varepsilon_{k, i_2} \varepsilon_{k,i_3} \, \sigma^4 \Big) + \frac{\sigma^4}4 \sum_{ \substack{i, i_1 \\ 1 < |i-i_1|< T}} \sum_{k=1}^n \varepsilon_{k,i} \varepsilon_{k, i_1} \nonumber \\
& + & \frac{\sigma^4}2  \sum_{ \substack{i, i_1, i_2 \\ 1 < |i-i_1|, \ldots, |i_2 -i|<T}} \sum_{k=1}^n \varepsilon_{k,i} \varepsilon_{k, i_1}  \varepsilon_{k, i_2} . \nonumber 
\end{eqnarray}
We use the inequality $ \ds\frac{x^2}2 - \frac{x^4}{12} \leq \log \cosh(x) \leq \frac{x^2}{2}$ for all $x \in \mathbb{R}$. Thus,  
\begin{eqnarray}
\bar{L}_{n,p,1} &:= & \sum_{\substack{ i < j \\ 1 < |i-j| <T}}   \log \cosh \Big( \sum_{k=1}^n \varepsilon_{k,i} \varepsilon_{k,j} X_{k,i} X_{k,j} \Big(  \sigma +  \sigma^3(1 + 2\sum_{\substack{ i_1 \\ 1 < |i_1 -i | <T} }    \varepsilon_{k,i_1}) \Big) \Big) \nonumber \\
&+&  \ds\frac{\sigma^2}2 \ds\sum_{\substack{i<j \\ 1 < |i-j| < T}} \sum_{k=1}^n \varepsilon_{k,i} \varepsilon_{k,j} -  \sigma^2  \sum_{ \substack{i < i_1 \\ 1 < |i_1 - i|< T}}  \ds\sum_{k=1}^n   X_{\varepsilon_{k,i}}^2\varepsilon_{k,i} \varepsilon_{k,i_1} \nonumber \\
& \leq & \ds\frac 12   \sum_{\substack{ i < j \\ 1 < |i-j| <T}}  \Big( \sum_{k=1}^n \varepsilon_{k,i} \varepsilon_{k,j} X_{k,i} X_{k,j} \Big(  \sigma +  \sigma^3(1 + 2 T ) \Big)^2 \nonumber \\
&+&  \ds\frac{\sigma^2}2 \ds\sum_{\substack{i<j \\ 1 < |i-j| < T}} \sum_{k=1}^n \varepsilon_{k,i} \varepsilon_{k,j} -  \sigma^2  \sum_{ \substack{i < i_1 \\ 1 < |i_1 - i|< T}}  \ds\sum_{k=1}^n   X_{\varepsilon_{k,i}}^2\varepsilon_{k,i} \varepsilon_{k,i_1} .
\end{eqnarray}
Therefore, 
\begin{eqnarray}
\mathbb{E}_\varepsilon \mathbb{E}_{I}^{( \varepsilon)} \Big( \bar{L}_{n,p,1} \Big) & \leq &  \ds\frac{na^2}2  \sum_{\substack{ i < j \\ 1 < |i-j| <T}}   \Big(  \sigma +  \sigma^3(1 + 2T ) \Big)^2 +  \ds\frac{ \, \sigma^2 }2 \sum_{\substack{ i < j \\ 1 < |i-j| <T}}  na^2  -   \sigma^2  \sum_{ \substack{i < i_1 \\ 1 < |i_1 - i|< T}} na^2 \nonumber \\\
& = & \ds\frac{na^2}2  \sum_{\substack{ i < j \\ 1 < |i-j| <T}}  \Big( 2 \sigma^4(1 +2T) + \sigma^6 (1 +2T)^2 \Big) = O(na^2 pT^2 \sigma^4 ) + O(na^2 pT^3 \sigma^6) \nonumber \\
&=& O(a^2 n p \varphi^{4} ) + O( a^2 np \varphi^6) = O \Big( (a^2 n)^{ \frac{- 4 \alpha +1}{4 \alpha +1}} p^{\frac 1{4 \alpha +1}} \Big) = o(1) \,  , 
\end{eqnarray}
as soon as $p = o(1) (a^2 n)^{4 \alpha -1}$. Similarly, we show that 
\begin{eqnarray}
\mathbb{E}_\varepsilon \mathbb{E}_{I}^{( \varepsilon)} \Big( \bar{L}_{n,p,1} \Big) & \geq &
- \ds\frac 1{12} \, \mathbb{E}_\varepsilon \mathbb{E}_{I}^{( \varepsilon)}  \sum_{\substack{ i < j \\ 1 < |i-j| <T}}  \Big( \sum_{k=1}^n \varepsilon_{k,i} \varepsilon_{k,j} X_{k,i} X_{k,j} \Big(  \sigma +  \sigma^3(1 + 2 \sum_{\substack{ i_1 \\ 1 < |i_1 -i | <T} } \varepsilon_{k,i}) \Big) \Big)^4  + o(1) \nonumber \\
&= & - \ds\frac{a^2n}{12}  \cdot  \mathbb{E}_{I}^{( \varepsilon_1)} (X_{k,i}^4 X_{k,j}^4 ) \sum_{\substack{ i < j \\ 1 < |i-j| <T}} \Big(  \sigma +  \sigma^3(1 + 2T ) \Big)^4  +  o(1)\nonumber \\
& -&  \ds\frac{a^4 n(n-1) }{12} \cdot 3  \mathbb{E}_{I}^{( \varepsilon_1)} (X_{1,i}^2  X_{1,j}^2 )  \mathbb{E}_{I}^{( \varepsilon_2)} (X_{2,i}^2  X_{2,j}^2 ) \sum_{\substack{ i < j \\ 1 < |i-j| <T}} \Big(  \sigma +  \sigma^3(1 + 2 T) \Big)^4 \nonumber.
\end{eqnarray}
See that the first term was already bounded from above in the previous display and that 
\begin{eqnarray*}
&& \ds\frac{a^4 n(n-1) }{12} \cdot 3  \mathbb{E}_{I}^{( \varepsilon_1)} (X_{1,i}^2  X_{1,j}^2 )  \mathbb{E}_{I}^{( \varepsilon_2)} (X_{2,i}^2  X_{2,j}^2 ) \sum_{\substack{ i < j \\ 1 < |i-j| <T}} \Big(  \sigma +  \sigma^3(1 + 2 T) \Big)^4 \nonumber \\
&\leq &  O(a^4 n^2 ) \cdot \Big( pT \sigma^4 + pT^2 \sigma^6  +  pT^3 \sigma^8 + pT^4 \sigma^{10} + pT^{5} \sigma^{12} \Big) = O(a^4 n^2 p \varphi^{4 + \frac 1{\alpha}}) = o(1) , 
\end{eqnarray*}
as soon  as $a^4 n^2 p \varphi^{4 + \frac 1{\alpha}} \to 0$ and $\alpha >1/2$.
We deduce that $\mathbb{E}_\varepsilon \mathbb{E}_{I}^{( \varepsilon)} \Big( \bar{L}_{n,p,1} \Big) \geq o(1)$. As consequence $\mathbb{E}_\varepsilon \mathbb{E}_{I}^{( \varepsilon)} \Big( \bar{L}_{n,p,1} \Big) = o(1)$. Now we treat the second term of $\bar{L}_{n,p}$:
\begin{eqnarray}
\bar{L}_{n,p, 2} &:= & \sum_{ \substack{ i,i_1,j \\ i<j \\ 1 < |i_1 - i|, |i_1 -j| < T}} \log \cosh \Big(  \ds\sum_{k=1}^n   X_{\varepsilon_{k,i}}  X_{\varepsilon_{k,j}} \varepsilon_{k,i} \varepsilon_{k,i_1} \varepsilon_{k,j} \, \sigma^2 \Big) \nonumber \\
& \leq & \sum_{ \substack{ i,i_1,j \\ i<j \\ 1 < |i_1 - i|, |i_1 -j| < T}} \Big(  \ds\sum_{k=1}^n   X_{\varepsilon_{k,i}}  X_{\varepsilon_{k,j}} \varepsilon_{k,i} \varepsilon_{k,i_1} \varepsilon_{k,j} \, \sigma^2 \Big) ^2 . \nonumber 
\end{eqnarray}
So, 
\begin{eqnarray*}
\mathbb{E}_\varepsilon \mathbb{E}_{I}^{( \varepsilon)} \Big( \bar{L}_{n,p, 2} \Big) &\leq &  \mathbb{E}_\varepsilon \mathbb{E}_{I}^{( \varepsilon)}  \sum_{ \substack{ i,i_1,j \\ i<j \\ 1 < |i_1 - i|, |i_1 -j| < T}} \Big(  \ds\sum_{k=1}^n   X_{\varepsilon_{k,i}}  X_{\varepsilon_{k,j}} \varepsilon_{k,i} \varepsilon_{k,i_1} \varepsilon_{k,j} \, \sigma^2 \Big) ^2 \nonumber \\
&=& \sum_{ \substack{ i,i_1,j \\ i<j \\ 1 < |i_1 - i|, |i_1 -j| < T}} \ na^3 \sigma^4 \leq a^3 n pT \sigma^4 = O(a^3 n p \varphi^{4 + \frac 1{\alpha}}) =o(1).
\end{eqnarray*}
Using the bound from below of $\log \cosh$ inequality, we show that $ \mathbb{E}_\varepsilon \mathbb{E}_{I}^{( \varepsilon)} \Big( \bar{L}_{n,p, 2} \Big)$  is bounded from below by a quantity that tends to zero. Therefore we get $\mathbb{E}_\varepsilon \mathbb{E}_{I}^{( \varepsilon)} \Big( \bar{L}_{n,p, 2} \Big) =o(1)$. In a similar way we show that the expected value of the remaining terms with $\log \cosh$  in $\bar{L}_{n,p}$ tend to~0.  Finally we have
\begin{eqnarray}
&& \mathbb{E}_\varepsilon \mathbb{E}_{I}^{( \varepsilon)} \Big(  \frac{\sigma^4}4 \sum_{ \substack{i, i_1 \\ 1 < |i-i_1|< T}} \sum_{k=1}^n \varepsilon_{k,i} \varepsilon_{k, i_1} +  \frac{\sigma^4}2  \sum_{ \substack{i, i_1, i_2 \\ 1 < |i-i_1|, \ldots, |i_2 -i|<T}} \sum_{k=1}^n \varepsilon_{k,i} \varepsilon_{k, i_1}  \varepsilon_{k, i_2} \Big) \nonumber \\
&=& O(a^2 \sigma^4 pTn ) +O( a^3\sigma^4 p T^2 n) =O(a^2 np \varphi^{4 + \frac 1\alpha} ) + O(a^3n p \varphi^4) = o(1), 
\end{eqnarray}
under the previous conditions.
As consequence  if $ p = o(1) (a^2 n)^{4 \alpha -1} $ and if $ a^4 n^2 p \varphi^{4 + \frac 1{\alpha}} \to 0$, then 
\begin{eqnarray}
\mathbb{E}_\varepsilon \mathbb{E}_{I}^{( \varepsilon)} \Big( - \bar{L}_{n,p} \Big) = o(1) \, .
\end{eqnarray}
To achieve the proof, we show in a similar way that $\mathbb{E}_\varepsilon \mathbb{E}_{I}^{( \varepsilon)} \Big( - \underline{L}_{n,p} \Big) = o(1)$ .
 \end{proof}


\noindent \begin{proof}[Proof of Theorem \ref{theo : adaptivityToepltz}]
To control the type I error probability,  we  derive an inequality of Berry-Essen type for $\mathcal{A}_{n,p, 2^l} $. For any fixed $l$ in $\mathbb{N}^*$ we denote by $v_{n,p,l} := \Var_{a, I}(\mathcal{A}_{n,p, 2^l} )$, which gives $v_{n,p,l} \sim a^4/(n^2 (p - 2^l)^2)$ by Proposition~\ref{prop : properties est toeplitz}. Next, we rewrite $\mathcal{A}_{n,p, 2^l}$ as follows :
\[
\mathcal{A}_{n,p, 2^l} = \underset{1 \leq k < \ell \leq n}{\ds\sum }H(Y_k, Y_\ell) 
\]
where,
\[
H(Y_k, Y_\ell) = \ds\frac{\ds\sqrt{2}}{n(n-1)(p-2^l)^2} \cdot \ds\frac 1{\ds\sqrt{ 2^l}}  \sum_{j=1}^{2^l}   ~ \underset{2^l+1 \leq i_1, i_2 \leq p}{\ds\sum } Y_{k, i_1} Y_{k, i_1-j} Y_{\ell, i_2} Y_{\ell, i_2-j}.
\]
 For $2 \leq k, h \leq n$,  define
\[
Z_k = \ds\frac{1}{\ds\sqrt{v_{n,p,l}}} \sum_{\ell=1}^{k-1} H(Y_k, Y_\ell) ,  \quad \text{and} \quad S_h = \sum_{k=2}^h Z_k  \,.
\]
Remark that $\{S_h \}_{h \geq 2}$ is a centered martingale with respect to the filtration $\{\mathcal{F}_h \}_{h \geq 2}$ where  $\mathcal{F}_h$ is the $\sigma$-field generated by the random vectors $\{ X_1, \dots, X_h \}$. Note that $ \mathcal{A}_{n,p, 2^l}  = \ds\sqrt{v_{n,p,l}} \cdot S_n$ and let 
$ V_n^2 = \sum_{k=2}^n \mathbb{E}_{a, I}(Z_k^2/ \mathcal{F}_{k-1}) $.
 We  fix $0 < \delta \leq 1$ and define
\[
J_n = \sum_{k=2}^n \mathbb{E}_{a,I}( Z_k)^{2 + 2 \delta} + \mathbb{E}_{a,I} | V_n^2 -1|^{1 + \delta} .
\] 
We use the Skorokhod representation and  Lemma 3.3 in \citep{HallHeyde} to obtain that, for any $0 < \varepsilon<1/2$ and any $x \in \mathbb{R}$, there exists a positive constant $C$ depending only on $\delta$ such that
\begin{eqnarray*}
|P_{a, I}(\mathcal{A}_{n,p, 2^l} \leq x) - \Phi \Big( \ds\frac x{\ds\sqrt{v_{n,p,l}}} \Big) | &=& \Big| P_{a, I} \Big(S_n \leq \ds\frac x{ \ds\sqrt{v_{n,p,l}}} \Big) - \Phi \Big( \ds\frac x{\ds\sqrt{v_{n,p,l}}} \Big) \Big| \\
& \leq &16 \varepsilon^{1/2} \exp \Big( - \ds\frac{x^2}{4 v_{n,p,l}} \Big)  + C \cdot \varepsilon^{-1- \delta} J_n.
\end{eqnarray*}
Then using that $ 1 - \Phi(u) \leq (1/u) \exp( -u^2/2)$ for all $u > 0$, we obtain
\begin{eqnarray}
 P_{a, I} (\mathcal{A}_{n,p, 2^l} > x) & \leq & \Big(1 -\Phi \Big( \ds\frac x{\ds\sqrt{v_{n,p,l}}} \Big) \Big)  + 16 \varepsilon^{1/2} \exp \Big( - \ds\frac{x^2}{4 v_{n,p,l}} \Big)  + C \cdot \varepsilon^{-1- \delta} J_n  \nonumber \\
 &\leq & \Big(  \ds\frac{1}{x/\ds\sqrt{v_{n,p,l}}} + 16 \varepsilon^{1/2} \Big)  \exp \Big( - \ds\frac{x^2}{4 v_{n,p,l}} \Big)  + C \cdot \varepsilon^{-1- \delta} J_n . \label{BerryEssen}
\end{eqnarray}
Choose $\delta =1$, then
\[
J_n = \sum_{k=2}^n \mathbb{E}_{a,I}( Z_k)^4 + \mathbb{E}_{a,I} | V_n^2 -1|^2.
\]
We can show that 
\begin{equation}
\label{J_nbound}
\sum_{k=2}^n \mathbb{E}_{a,I}( Z_k)^4  = O \Big( \ds\frac{1}{n} \Big) \quad \text{and} \quad \mathbb{E}_{a,I} | V_n^2 -1|^2 =O \Big( \ds\frac{1}{n} \Big) + O \Big( \ds\frac{1}{2^l} \Big)
\end{equation}
 Take $t_l = a^2\ds\frac{ \ds\sqrt{\mathcal{C}^* \ln l}}{n(p-2^l)}$, we use \eqref{BerryEssen} and \eqref{J_nbound} to bound from above the type I error probability:
\begin{eqnarray*}
P_{a, I}( \Delta_{ad} =1 ) &=& P_{a, I}(\exists l \in \mathbb{N}, L_* \leq  l \leq L^* \, ; \mathcal{A}_{n,p, 2^l} > t_l) \leq \ds\sum_{L_* \leq  l \leq L^*} P_{a, I}( \mathcal{A}_{n,p, 2^l}  > t_l) \nonumber \\
& \leq &  \ds\sum_{L_* \leq  l \leq L^*} \left( \Big( \frac {a^2}{n(p-2^l) t_l} + 16 \varepsilon^{1/2} \Big) \exp \Big( -  \ds\frac{t_l^2}{4 v_{n,p}}  \Big) + \frac{O(1)}{\varepsilon^{2}} \Big(\ds\frac 1n +  \frac 1{2^l}\Big) \right) \nonumber \\
& \leq &  \sum_{l \geq L_*} \Big( \ds\frac 1{ \ds\sqrt{\mathcal{C}^*\ln l}} + 16 \varepsilon^{1/2} \Big) \exp \Big( - \ds\frac{ \mathcal{C}^* \ln l}{4} \Big) + O(1) \ds\frac{L^* -L_*}{n \varepsilon^2} + \frac{O(1)}{\varepsilon^2} \sum_{l \geq L_*} \ds\frac{1}{2^l} 
 \nonumber \\
& \leq & \sum_{l \geq L_*} \Big( \ds\frac{1}{  \ds\sqrt{\mathcal{C}^* \ln l}} + 16 \varepsilon^{1/2} \Big)l^{-\mathcal{C}^*/4} + + O(1) \ds\frac{L^* }{n \varepsilon^2} + \frac{O(1) 2^{-L_*}}{\varepsilon^2} = o(1),
\end{eqnarray*}
 for $\mathcal{C}^* > 4$ and since $ L_*$ and $L^*$ both tend to infinity, such that $\ln(a^2 n \sqrt{p})/n $ tends to 0.

Now, we control the type II error probability. Assume that $\Sigma \in \mathcal{T}(\alpha)$ and that $\alpha$ is such that  there exists $l_0 \in \{L_*, \ldots,L^*\}$ such that $2^{l_0 -1} \leq    (\psi_\alpha)^{- \frac 1{\alpha}}  < 2^{l_0} $, thus  
\begin{eqnarray}
\mathbb{E}_{a,\Sigma}(\widehat{\mathcal{D}}_{n,p,2^{l_0}}) &=&    \frac{a^4}{ \ds\sqrt{2 \cdot 2^{l_0}}} \Big( \ds\sum_{ 1 \leq j < p }  \sigma_{j}^2 -  \ds\sum_{ 2^{l_0}<j < p }  \sigma_{j}^2 \Big) \nonumber \\
& \geq & \frac{a^4 }{2 (\psi_\alpha)^{- \frac 1{ 2 \alpha}} } \Big( \mathcal{C}^2 \psi_\alpha^2 - \sum_{j} \ds\frac{ j^{2\alpha}}{(2^{l_0})^{2 \alpha}}  \sigma_{j}^2 \Big) \geq  (\psi_\alpha)^{ 2 + \frac 1{2\alpha}} \cdot \frac{a^4}{ 2} \Big( \mathcal{C}^2 -1 \Big). \nonumber 
\end{eqnarray}
We assumed that $a^2 np (\psi_\alpha)^{2 + \frac{1}{2 \alpha}} =  \ds\sqrt{ \ln \ln (a^2 np)}$. Moreover, we have 
$$
t_{l_0} \leq \frac{a^2 \ds\sqrt{\mathcal{C}^* \ln L^* }}{n (p-2^{l_0})} \leq \frac{a^2 \ds\sqrt{ \mathcal{C}^* \ln \ln(a^2np) }}{n (p-2^{l_0})} \leq 2 \ds\sqrt{ \mathcal{C}^* } a^4 (\psi_\alpha)^{2 + \frac{1}{2 \alpha}}.
$$
Thus, we have $\mathbb{E}_{a, \Sigma}(\mathcal{A}_{n,p, 2^{l_0}})- t_{l_0} \geq a^4 (\psi_\alpha)^{2+1/(2 \alpha)} (\mathcal{C}^2 - 1 - 4 \ds\sqrt{\mathcal{C}^*})/2$  by our assumption that $ \mathcal{C}^2 > 1 + 4 \ds\sqrt{\mathcal{C}^*} $. 
Therefore we get 
\begin{eqnarray*}
P_{a, \Sigma} ( \Delta_{ad} =0) &=& P_{a, \Sigma}(\forall l \in \{L_*,\ldots, L^*\}\, ; \mathcal{A}_{n,p, 2^l} < t_l) \leq P_{a, \Sigma}( \mathcal{A}_{n,p, 2^{l_0}} < t_{l_0})  \nonumber \\
& \leq &  P_{a, \Sigma}( | \mathcal{A}_{n,p, 2^{l_0}} - \mathbb{E}_{a, \Sigma}(\mathcal{A}_{n,p, 2^{l_0}}) | >  \mathbb{E}_{a, \Sigma}(\mathcal{A}_{n,p, 2^{l_0}})- t_{l_0}). 
\end{eqnarray*}
Now, we use Markov inequality and get :
\begin{eqnarray}
P_{a, \Sigma} ( \Delta_{ad} =0)  & \leq & \ds\frac{\Var_{a, \Sigma}(\mathcal{A}_{n,p, 2^{l_0}})}{( \mathbb{E}_{a, \Sigma}(\mathcal{A}_{n,p, 2^{l_0}})- t_{l_0})^2} \nonumber \\
& \leq & \ds\frac{R_1 + (n-1)(p-2^{l_0})^2R_2}{n(n-1)(p-2^{l_0})^4 ( \mathbb{E}_{a, \Sigma}(\mathcal{A}_{n,p, 2^{l_0}})- t_{l_0})^2}.  \label{erreur de seconde espece} 
\end{eqnarray}
First, we bound from above the first term of \eqref{erreur de seconde espece}, using Proposition~\ref{prop : properties est toeplitz}
\begin{eqnarray*}
S_1 &:=& \ds\frac{R_1}{n(n-1)(p-2^{l_0})^4 ( \mathbb{E}_{a, \Sigma}(\mathcal{A}_{n,p, 2^{l_0}})- t_{l_0})^2}  \nonumber \\ 
&=& \ds\frac{ a^4(1 + o(1)) +  \mathbb{E}_\Sigma(\widehat{\mathcal{A}}_{n, p, 2^{l_0}}) \cdot ( O(a^2\ds\sqrt{2^{l_0}}) + O( a^3(2^{l_0})^{3/2 - 2 \alpha}))  }{n(n-1)(p-2^{l_0})^2 ( \mathbb{E}_{a, \Sigma}(\mathcal{A}_{n,p, 2^{l_0}})- t_{l_0})^2}  \nonumber \\
&+ & \ds\frac{ \mathbb{E}^2_\Sigma(\widehat{\mathcal{A}}_{n, p, 2^{l_0}}) \cdot O(m^2/a) }{n(n-1)(p-2^{l_0})^2 ( \mathbb{E}_{a, \Sigma}(\mathcal{A}_{n,p, 2^{l_0}})- t_{l_0})^2} 
\end{eqnarray*}
We decompose $S_1$  as sum of three terms: the first one 
\begin{eqnarray}
S_{1,1}&:=& \ds\frac{  a^4(1 + o(1)) }{n(n-1)(p-2^{l_0})^2 ( \mathbb{E}_{a, \Sigma}(\mathcal{A}_{n,p, 2^{l_0}})- t_{l_0})^2} \nonumber \\
& \leq & \ds\frac{a^4(1+ o(1))}{n(n-1)(p-2^{l_0})^2  a^8 (\psi_\alpha)^{4 + \frac{1}{\alpha}} \Big( \mathcal{C}^2 -1 - 4 \ds\sqrt{\mathcal{C}^*}  \Big)^2} \nonumber \\
&=& O \Big(  \ds\frac{1}{\ln \ln(a^2np)} \Big) =o(1) \, , \text{ as soon as } a^2np \to + \infty .  \nonumber 
\end{eqnarray}
Now we show that the second term of $S_1$ also tends to 0. Recall that $2^{l_0} \asymp (\psi_\alpha)^{- \frac 1{\alpha}}$, therefore
\begin{eqnarray}
S_{1,2} &:=&  \ds\frac{ O(a^2 \,  2^{l_0} ) + O( a^3(2^{l_0})^{3/2 - 2 \alpha})}{n(n-1)(p-2^{l_0})^2  \mathbb{E}_{a, \Sigma}(\mathcal{A}_{n,p, 2^{l_0}})(1- t_{l_0}/ \mathbb{E}_{a, \Sigma}(\mathcal{A}_{n,p, 2^{l_0}}))^2} \nonumber \\[0.4 cm]
& \leq & \ds\frac{ (O(a^2 \,  2^{l_0}) + O( a^3(2^{l_0})^{3/2 - 2 \alpha}))}{n(n-1)(p-2^{l_0})^2  \mathbb{E}_{a, \Sigma}(\mathcal{A}_{n,p, 2^{l_0}}) \Big( 1- \ds\frac{4 \ds\sqrt{ \mathcal{C}^*}}{\mathcal{C}^2 -1} \Big)^2}  \nonumber \\[0.4 cm]
& \leq & \ds\frac{ O(2^{l_0}) + O( a(2^{l_0})^{3/2 - 2 \alpha})}{n(n-1)(p-2^{l_0})^2  a^2 (\psi_\alpha)^{ 2 + \frac 1{2 \alpha}} \Big(1- \ds\frac{4 \ds\sqrt{ \mathcal{C}^*}}{\mathcal{C}^2 -1} \Big)^2}  \nonumber \\[0.4 cm]
& \leq & \ds\frac{O(\ds\sqrt{2^{l_0}} \cdot (\psi_\alpha)^{ 2 + \frac 1{2 \alpha}} ) + O( (2^{l_0})^{3/2 - 2 \alpha} \cdot (\psi_\alpha)^{ 2 + \frac 1{2 \alpha}} )}{  \ln\ln (a^2np) } =o(1). \nonumber 
\end{eqnarray}
since $2^{l_0} \cdot (\psi_\alpha)^{ 2 + \frac 1{2 \alpha}} \asymp(\psi_\alpha)^{2 - \frac 1{2\alpha}}  =o(1)$   and $  (2^{l_0})^{3/2 - 2 \alpha} \cdot  (\psi_\alpha)^{ 2 + \frac 1{2 \alpha}} \asymp (\psi_\alpha)^{4 - \frac 1{ \alpha}} =o(1) $ for all $\alpha > 1/4$.
Finally,
\begin{eqnarray*}
S_{1,3} &:=&  \ds\frac{ \mathbb{E}^2_\Sigma(\widehat{\mathcal{A}}_{n, p, 2^{l_0}}) \cdot O((2^{l_0})^2) }{n(n-1)(p-2^{l_0})^2 ( \mathbb{E}_{a, \Sigma}(\mathcal{A}_{n,p, 2^{l_0}})- t_{l_0})^2}  \nonumber \\
& =& \ds\frac{O( (2^{l_0})^2)}{ n(n-1) p^2  }=o(1).
\end{eqnarray*}
Now, we bound from above the second term of \eqref{erreur de seconde espece}.
\begin{eqnarray*}
S_2 & =&  \ds\frac{R_2}{n(p-2^{l_0})^2( \mathbb{E}_{a, \Sigma}(\mathcal{A}_{n,p, 2^{l_0}})- t_{l_0})^2} \nonumber \\
&=&  \ds\frac{ a^2 \cdot \mathbb{E}_{\Sigma} (\widehat{\mathcal{A}}_{n, p,2^{l_0}}) \cdot o(1) }{n(p-2^{l_0}) (\mathbb{E}_{a, \Sigma}(\mathcal{A}_{n,p, 2^{l_0}})- t_{l_0})^2}  
+ \frac{ \mathbb{E}^2_{\Sigma}(\widehat{\mathcal{A}}_{n, p, 2^{l_0}}) \cdot O(2^{l_0}) }{n(p-2^{l_0}) (\mathbb{E}_{a, \Sigma}(\mathcal{A}_{n,p, 2^{l_0}})- t_{l_0})^2}  \nonumber \\[0.3 cm]
&+&  \frac{\mathbb{E}^{3/2}_{\Sigma}(\widehat{\mathcal{A}}_{n, p, 2^{l_0}} ) \cdot (O( a\cdot (2^{l_0})^{1/4})+ O(a^2 (2^{l_0})^{3/4- \alpha}))}{n(p-2^{l_0}) (\mathbb{E}_{a, \Sigma}(\mathcal{A}_{n,p, 2^{l_0}})- t_{l_0})^2}.
\end{eqnarray*}
 We bound from above each term  of $S_2$. For the first term,
 \begin{eqnarray}
 S_{2,1} &:=&  \ds\frac{ a^2 \cdot \mathbb{E}_{\Sigma} (\widehat{\mathcal{A}}_{n, p,2^{l_0}}) \cdot o(1) }{n(p-2^{l_0}) (\mathbb{E}_{a, \Sigma}(\mathcal{A}_{n,p, 2^{l_0}})- t_{l_0})^2} 
  \leq \ds\frac{ o(1)}{n(p-2^{l_0})  a^2 (\psi_\alpha)^{ 2 + \frac 1{2 \alpha}} }  \nonumber \\
 & = &  \ds\frac{o(1)}{\ds\sqrt{\ln \ln (a^2 np)}} = o(1). \nonumber 
 \end{eqnarray}
For the second term we have,
\begin{eqnarray*}
S_{2,2} &:=& \frac{ \mathbb{E}^2_{\Sigma}(\widehat{\mathcal{A}}_{n, p, 2^{l_0}}) \cdot O(2^{l_0}) }{n(p-2^{l_0}) (\mathbb{E}_{a, \Sigma}(\mathcal{A}_{n,p, 2^{l_0}})- t_{l_0})^2}  \leq  \ds\frac{ O(2^{l_0}) }{n p } = o(1) 
\end{eqnarray*}
Finally for the last term, 
\begin{eqnarray*}
S_{2,3} &:=& \frac{\mathbb{E}^{3/2}_{\Sigma}(\widehat{\mathcal{A}}_{n, p, 2^{l_0}} ) \cdot (O( a\cdot (2^{l_0})^{1/4})+ O(a^2 (2^{l_0})^{3/4- \alpha}))}{n(p-2^{l_0}) (\mathbb{E}_{a, \Sigma}(\mathcal{A}_{n,p, 2^{l_0}})- t_{l_0})^2} \nonumber \\
& \leq & \frac{O(  (2^{l_0})^{1/4})+ O(a \cdot  (2^{l_0})^{3/4- \alpha})}{n(p-2^{l_0}) a \psi_\alpha^{1 +  \frac{1}{4\alpha} } } \nonumber \\
& \leq &   \frac{O(  (2^{l_0})^{1/4})}{ \ds\sqrt{n(p-2^{l_0})} \,  (\ln \ln(a^2 np))^{\frac 14} } +  \frac{ O(a^2  \cdot \psi_\alpha^{1 +\frac 1{4 \alpha}} \cdot  (2^{l_0})^{3/4- \alpha})}{\ds\sqrt{\ln \ln(a^2 np)}} = o(1),
\end{eqnarray*}
as $ a^2  \cdot \psi^{1 +\frac 1{4 \alpha}} \cdot  (2^{l_0})^{3/4- \alpha} \asymp \psi_\alpha^{2 - \frac 1{2 \alpha} } =o(1)$.
\hfill \end{proof}

\bibliographystyle{plain}
\bibliography{biblio123}

\end{document}